\documentclass{amsproc}

\usepackage[utf8]{inputenc}

\usepackage[centertags]{amsmath}
\usepackage{amsfonts}
\usepackage{amssymb}

\usepackage{amscd}
\usepackage{mathtools}
\usepackage{xcolor}
\definecolor{darkblue}{rgb}{0,0,0.5}
\definecolor{darkgreen}{rgb}{0,0.5,0}
\usepackage[colorlinks=true,linkcolor=darkblue,urlcolor=darkblue,citecolor=darkgreen]{hyperref}
\usepackage{algorithmic, algorithm}
\usepackage{graphicx}
\usepackage{todonotes}
\usepackage{color}
\usepackage{pgfplots}
\usepackage{faktor}
\usepackage{mathrsfs} 
\usepackage[capitalize]{cleveref}
\usepackage{subfig}
\usepackage{tikz}
\usepackage[framemethod=tikz]{mdframed}
\usepackage{tabularx,booktabs}
\usetikzlibrary{shadows.blur, shapes.arrows}
\usepackage[shortlabels]{enumitem}
\numberwithin{equation}{section}
\usepackage{xspace}
\usepackage{colonequals}
\renewcommand{\contentsname}{Contents}

\def\hmath$#1${\texorpdfstring{{\rmfamily\textit{#1}}}{#1}}

\definecolor{green}{RGB}{34,139,34}

\newcommand{\OrdersToJBig}[1]{{\textsf{\upshape Orders\-To\-j\-Invariant\-Big\-Set}\ensuremath{}}}
\newcommand{\OrdersToJSmall}{{\textsf{\upshape Orders\-To\-jInvariant\-Small\-Set}}}
\newcommand{\OrderToJ}[1]{{\textsf{\upshape Single\-Order\-To\-j\-Invariant}\ensuremath{(#1)}}}

\newcommand{\OrderToJName}{{\textsf{\upshape Single\-Order\-To\-jInvariant}}}

\newcommand{\KerToId}[1]{\hyperref[alg: kernel to ideal]{\textsf{\upshape Kernel\-To\-Ideal}\ensuremath{_{#1}}}}
\newcommand{\IdToKer}[1]{\hyperref[alg: ideal to kernel]{\textsf{\upshape Ideal\-To\-Kernel}\ensuremath{_{#1}}}}

\newcommand{\ItIE}[1]{\hyperref[alg: new algo]{\textsf{\upshape Ideal\-To\-Isogeny\-From\-Eichler}\ensuremath{_{#1}}}}
\newcommand{\IdToIsoEich}[1]{\hyperref[alg: new algo]{\textsf{\upshape Ideal\-To\-Isogeny\-From\-Eichler}\ensuremath{_{#1}}}}
\newcommand{\IdToIsoKLPT}[1]{\hyperref[alg:IdealToIsogeny2bullet]{\textsf{\upshape Ideal\-To\-Isogeny\-From\-KLPT}\ensuremath{_{#1}}}}
\newcommand{\IdToIsoCoprime}[1]{\hyperref[alg:SpecialIdealToIsogeny]{\textsf{\upshape Ideal\-To\-Isogeny\-Coprime}\ensuremath{_{#1}}}}
\newcommand{\IdToIsoSmallKLPT}[1]{\hyperref[alg:IdealToIsogeny2eepsilon]{\textsf{\upshape Ideal\-To\-Isogeny\-Small\-From\-KLPT}\ensuremath{_{#1}}}}
\newcommand{\IdToIsoSmallEich}[1]{\hyperref[alg: new sub-algo]{\textsf{\upshape Ideal\-To\-Isogeny\-Small\-From\-Eichler}\ensuremath{_{#1}}}}
\newcommand{\SpEicherN}[1]{\hyperref[alg: special eichler norm eq]{\textsf{\upshape Special\-Eichler\-Norm}\ensuremath{_{#1}}}}
\newcommand{\EichlerN}[1]{\hyperref[alg: eichler norm eq]{\textsf{\upshape Eichler\-Norm}\ensuremath{_{#1}}}}
\newcommand{\IdEichlerN}[1]{\hyperref[alg: eichler ideal norm equation]{\textsf{\upshape Ideal\-Eichler\-Norm}\ensuremath{_{#1}}}}
\newcommand{\SubEichlerN}[1]{\hyperref[alg: extended eichler]{\textsf{\upshape Suborder\-Eichler\-Norm}\ensuremath{_{#1}}}}
\newcommand{\IdSubEichlerN}[1]{\hyperref[alg: ideal suborder norm eq]{\textsf{\upshape Ideal\-Suborder\-Eichler\-Norm}\ensuremath{_{#1}}}}

\newcommand{\FRI}[1]{\hyperref[alg: repres integer full]{\textsf{\upshape Full\-Represent\-Integer}\ensuremath{_{#1}}}}
\newcommand{\RI}[1]{\hyperref[alg: repres integer]{\textsf{\upshape Represent\-Integer}\ensuremath{_{#1}}}}
\newcommand{\FSA}[1]{\hyperref[alg: full strong approx]{\textsf{\upshape Full\-Strong\-Approximation}\ensuremath{_{#1}}}}
\newcommand{\SA}[1]{\hyperref[alg:strong approx]{\textsf{\upshape Strong\-Approximation}\ensuremath{_{#1}}}}
\newcommand{\KLPT}[1]{\hyperref[alg:KLPT]{\textsf{\upshape KLPT}\ensuremath{_{#1}}}}
\newcommand{\GenKLPT}[1]{\hyperref[alg: generic klpt]{\textsf{\upshape Generic\-KLPT}\ensuremath{_{#1}}}}
\newcommand{\SignKLPT}[1]{\hyperref[alg: sign klpt]{\textsf{\upshape Signing\-KLPT}\ensuremath{_{#1}}}}

\newcommand{\SmoothGen}[1]{\hyperref[alg: smooth basis]{\textsf{\upshape Generating\-Family}\ensuremath{_{#1}}}}

\newcommand{\Compress}[1]{\hyperref[alg:compression]{\textsf{\upshape Compression}\ensuremath{_{#1}}}}
\newcommand{\Decompress}[1]{\hyperref[alg:decompression]{\textsf{\upshape Decompression}\ensuremath{_{#1}}}}

\newcommand{\EndoEval}[1]{\hyperref[alg: endo eval]{\textsf{\upshape Endomorphism\-Evaluation}\ensuremath{_{#1}}}}
\newcommand{\IdealEvaluation}[1]{\hyperref[alg: ideal evaluation]{\textsf{\upshape Ideal\-Evaluation}\ensuremath{_{#1}}}}
\newcommand{\SuborderEvaluation}[1]{\hyperref[alg: suborder evaluation]{\textsf{\upshape Suborder\-Evaluation}\ensuremath{_{#1}}}}
\newcommand{\VerifIdeal}[1]{\hyperref[alg: verif ideal proof]{\textsf{\upshape Ideal\-Verification}\ensuremath{_{#1}}}}

\newcommand{\VerifSuborder}[1]{\hyperref[alg: verif suborder proof]{\textsf{\upshape Suborder\-Verification}\ensuremath{_{#1}}}}
\newcommand{\IdealToSuborder}[1]{\hyperref[alg: ideal to suborder]{\textsf{\upshape Ideal\-To\-Suborder}\ensuremath{_{#1}}}}
\newcommand{\CheckTrace}[1]{\hyperref[alg: check trace]{\textsf{\upshape Check\-Trace}\ensuremath{_{#1}}}}
\newcommand{\InverseTrapdoor}[1]{\hyperref[alg:inv]{\textsf{\upshape Inverse\-Trapdoor}\ensuremath{_{#1}}}}

\newcommand{\pSIDHKeyGen}[1]{\hyperref[alg: keygen]{\textsf{\upshape pSIDH\-Key\-Gen}\ensuremath{_{#1}}}}
\newcommand{\pSIDHKeyEx}[1]{\hyperref[alg: key exch]{\textsf{\upshape pSIDH\-Key\-Exchange}\ensuremath{_{#1}}}}
\newcommand{\SetaQuadOrder}[1]{\hyperref[alg: keygen param]{\textsf{\upshape Seta\-Quadratic\-Order\-Gen}\ensuremath{_{#1}}}}
\newcommand{\SetaCurveGen}[1]{\hyperref[alg:randomized trapdoor curve gen]{\textsf{\upshape Seta\-Curve\-Gen}\ensuremath{_{#1}}}}

\newcommand{\SupersingularEvaluation}[1]{\hyperref[alg: supersingular evaluation]{\textsf{\upshape Super\-singular\-Eval\-ua\-tion}\ensuremath{_{#1}}}}
\newcommand{\SpecialSupersingularEvaluation}[1]{\hyperref[alg: special supersingular evaluation]{\textsf{\upshape Special\-Super\-singular\-Eval\-ua\-tion}\ensuremath{_{#1}}}}
\newcommand{\ModularEvaluationBigLevel}[1]{\hyperref[alg: modular evaluation big level]{\textsf{\upshape Mo\-dular\-Eval\-ua\-tion\-Big\-Level}\ensuremath{_{#1}}}}
\newcommand{\ModularEvaluationBigChar}[1]{\hyperref[alg: modular evaluation big char]{\textsf{\upshape Mo\-dular\-Eval\-ua\-tion\-Big\-Char\-acteristic}\ensuremath{_{#1}}}}
\newcommand{\ModularEvaluation}[1]{\hyperref[alg: modular evaluation]{\textsf{\upshape Mo\-dular\-Eval\-ua\-tion}\ensuremath{_{#1}}}}
\newcommand{\ModularComputation}[1]{{\textsf{\upshape Mo\-dular\-Com\-puta\-tion}\ensuremath{_{#1}}}}
\newcommand{\GetWeberCube}{\textsf{\upshape GetWeberCube}}

\newcommand{\EndoRing}[1]{\hyperref[alg: endo ring]{\textsf{\upshape Endo\-morphism\-Ring}\ensuremath{_{#1}}}}
\newcommand{\SupersingularModularComputation}[1]{{\textsf{\upshape Super\-singular\-Modular\-Compu\-ta\-tion}\ensuremath{_{#1}}}}

\newcommand{\llog}{\textnormal{llog}}

\newcommand{\primesell}{\mathcal{P}_\ell}

\newcommand{\primesdelta}{\mathcal{P}_\Delta}

\newcommand*{\frakl}{\mathfrak{l}}
\newcommand{\fraka}{\mathfrak{a}} 
\newcommand{\frakS}{\mathfrak{S}} 
\newcommand{\frakb}{\mathfrak{b}}

\renewcommand{\O}{\mathcal{O}}

\newcommand{\QA}{\mathcal{B}_{p,\infty}}

\newcommand{\FF}{\mathbb{F}}
\newcommand{\ZZ}{\mathbb{Z}}

\newcommand{\QQ}{\mathbb{Q}}
\newcommand{\PP}{\mathbb{P}}
\newcommand{\CC}{\mathbb{C}}

\newcommand{\frakO}{\mathfrak{O}}

\newcommand{\frakf}{\mathfrak{f}}
\newcommand{\Cl}{\textnormal{Cl}}

\newcommand{\polylog}[1]{ O(\mathsf{poly} (\log (#1) ))}

\definecolor{light blue}{RGB}{0,102,204}

\renewcommand{\phi}{\varphi}

\DeclareFontFamily{OT1}{rsfs}{}
\DeclareFontShape{OT1}{rsfs}{n}{it}{<-> rsfs10}{}
\DeclareMathAlphabet{\mathscr}{OT1}{rsfs}{n}{it}

\newcommand{\fa}{\mathfrak{a}}

\renewcommand{\L}{\mathcal{L}}

\newcommand{\comment}[1]{}

\newcommand{\Fpbar}{\overline{\FF}_p}

\newcommand{\Fp}{\FF_p}

\renewcommand{\O}{\mathcal{O}}

\renewcommand\theparagraph{\arabic{section}.\arabic{subsection}.\arabic{paragraph}}

\DeclareMathOperator{\tr}{tr}

\DeclareMathOperator{\End}{End}

\newtheorem{claime}{Heuristic}

\crefname{figure}{Figure}{Figures}
\crefname{problem}{Problem}{Problems}
\crefname{proposition}{Proposition}{Propositions}
\crefname{algorithm}{Algorithm}{Algorithms}
\crefname{step}{Step}{Steps}
\crefname{lemma}{Lemma}{Lemmas}
\crefname{definition}{Definition}{Definitions}
\crefname{claime}{Heuristic}{Heuristics}

\crefname{assumption}{Assumption}{Assumptions}
\crefname{conjecture}{Conjecture}{Conjectures}

\newcounter{tasknumber}
\newcommand{\task}[2][]{%
  \addtocounter{tasknumber}{1}%
  \begin{center}%
  \framebox[1.1\width]{\begin{minipage}{0.9\textwidth}%
  \textbf{Task \arabic{tasknumber}} \textit{\if!#1(unassigned)!\else (#1)\fi}: {#2}%
  \end{minipage}}%
  \end{center}%
}

\pgfplotsset{compat=1.17}

\newcommand\Require{\REQUIRE}
\newcommand\Return{\RETURN}
\newcommand\Ensure{\ENSURE}
\renewcommand{\emph}[1]{\textit{#1}}

\newtheorem{theorem}{Theorem}[section]

\newtheorem{proposition}[theorem]{Proposition}

\theoremstyle{definition}

\theoremstyle{remark}
\newtheorem{remark}[theorem]{Remark}

\numberwithin{equation}{section}

\begin{document}

\title[Eval. of Modular Poly. from Supersingular Ell. Curves]{Evaluation of Modular Polynomials from Supersingular Elliptic Curves}


\author[M. C.-Real Santos]{Maria Corte-Real Santos}
\address{ENS de Lyon, CNRS, UMPA, UMR 5669, Lyon, France}
\email{maria.corte\_real\_santos@ens-lyon.fr}

\author[J. K. Eriksen]{Jonathan Komada Eriksen}
\address{COSIC, KU Leuven, Belgium}
\email{jeriksen@esat.kuleuven.be}

\author[A. Leroux]{Antonin Leroux}
\address{DGA-MI, Bruz, France 
 \\ IRMAR - UMR 6625, Université de Rennes, France}
\email{antonin.leroux@polytechnique.org}

\author[M. Meyer]{Michael Meyer}
\address{University of Regensburg, Germany}
\email{michael@random-oracles.org}

\author[L. Panny]{Lorenz Panny}
\address{Technische Universität München, Germany}
\email{lorenz@yx7.cc}

\begin{abstract}
  We present several new algorithms to evaluate modular polynomials of level $\ell$ modulo a prime $p$ on an input $j$. 
    More precisely, we introduce two new generic algorithms, sharing the following similarities: they are based on a CRT approach; they make use of supersingular curves and the Deuring correspondence; and, their memory requirements are optimal. 

    The first algorithm combines the ideas behind a hybrid algorithm of Sutherland in 2013 with a recent algorithm to compute modular polynomials using supersingular curves introduced in 2023 by Leroux. The complexity (holding around several plausible heuristic assumptions) of the resulting algorithm matches the $\Tilde{O}(\ell^3 \log^{3} \ell + \ell \log p)$ time complexity of the best known algorithm by Sutherland, but has an optimal memory requirement.    
    
    Our second algorithm is based on a sub-algorithm that can evaluate modular polynomials efficiently on supersingular $j$-invariants defined over $\Fp$, and achieves heuristic complexity quadratic in both $\ell$ and $\log j$, and linear in $\log p$. In particular, it is the first generic algorithm with optimal memory requirement to obtain a quadratic complexity in~$\ell$. 
    
    Additionally, we show how to adapt our method to the computation of other types of modular polynomials such as the one stemming from Weber's function.  
    
    Finally, we provide an optimised implementation of the two algorithms detailed in this paper, though we emphasise that various modules in our codebase
    may find applications outside their use in this paper.
\end{abstract}

\maketitle




\section{Introduction}

The evaluation of modular polynomials is one of the main subroutines involved in the SEA point counting algorithm~\cite{elkies1998elliptic}. The main application of point counting is to find elliptic curves suitable for cryptography. The asymptotic bottleneck of the computation relating to ``Elkies primes'' in the SEA algorithm is modular evaluation, and it is thus important to find improvements to the theoretical and practical efficiency of modular evaluation.

More generally, modular polynomials play an important role in the theory of elliptic curves, and so the computation and evaluation of modular polynomials is a central task in algorithmic number theory. Aside from point counting, modular polynomials are related to isogeny computations. While most applications tend to use the more efficient Vélu formulas~\cite{V71, BFLS20}, we still observe few instances where modular polynomials have been considered. For example, it is used in the CRS key exchange~\cite{couveignes,RS06}, the first isogeny-based protocol, the OSIDH construction~\cite{CK19}, the reduction introduced in~\cite{arpin2023finding}, or the fast implementation of Delfs--Galbraith proposed in \cite{corte2022accelerating}. 

In this work, we introduce several new algorithms for efficiently computing modular polynomials $\Phi_\ell$ of level $\ell$, evaluated in one of the variables.

\vspace{1ex}

\subsubsection*{Related work.} 
Almost all known methods to efficiently evaluate modular polynomials first requires their computation, which is the primary reason most literature on the topic focuses on the computation of modular polynomials. 

The historical approach to computing modular polynomials is based on the computation of the coefficients of the Fourier expansion of the modular $j$-function \cite{blake1999computation,elkies1998elliptic,morain1995calcul,lehmann1994counting}. As this computation works over the integers, it can be applied to compute $\Phi_\ell \bmod p$. 

An alternative approach based on the CRT method, which uses supersingular curves and $\ell$-isogeny computations, was introduced by Charles and Lauter \cite{charles2005computing} to work entirely over $\FF_p$.  

Enge \cite{enge2009complexity} uses interpolation and fast floating-point interpolation to obtain a quasi-linear algorithm over the integers under some heuristics. 

One of the main problems behind the computation of modular polynomials over the integers is the large size of their coefficients.  
In theory, the size of coefficients is less of an issue when dealing with $\Phi_{\ell}$ as the coefficients are reduced modulo $p$. However, this is only the case if we can avoid the computation over $\ZZ$ entirely, which is not easy to obtain in practice. 

Bröker, Lauter, and Sutherland (BLS hereafter) \cite{broker2012modular} obtained
the first quasi-linear complexity in both time and space with a careful application of the CRT method using ordinary curves. 

More recently, Leroux \cite{leroux2023computation} revisited the CRT approach from supersingular curves with new algorithmic results on the Deuring correspondence to obtain an algorithm with the same complexity as BLS. Leroux claims a better practical efficiency, but provided no experimental proof of this claim. 

Robert \cite{robert2022some} outlined another CRT method based on supersingular curves by using a $p$-adic approach in conjunction with the high-dimensional isogeny technique introduced in the context of the cryptanalysis of the SIDH key exchange.
This method was later refined and implemented by Kunzweiler and Robert \cite{kunzweiler2024computing}. However, the main advantage of their method is mostly theoretical as it does not require any assumptions (not even GRH!), but it is unlikely to be practical due to a non-optimal high memory requirement. Furthermore, the implementation they provide is in SageMath \cite{sage}, rather than a low-level implementation.     



In 2013, Sutherland \cite{sutherland2013evaluation} provided the first evidence that a method tailored to the evaluation could do better than a computation algorithm. Sutherland showed how to adapt the CRT approach from BLS to obtain two evaluation algorithms with excellent memory requirements. However, these algorithms have essentially the same complexity as the BLS computation algorithm \cite{broker2012modular}. 
 
\vspace{1ex}

\subsubsection*{Contributions.}
Our main contribution is the introduction of three new algorithms targeted at the task of evaluating modular polynomials. Our work builds upon the previous idea of Leroux \cite{leroux2023computation} to use supersingular curves with the Deuring correspondence for efficient modular polynomial computation. The fact that his algorithms could be improved to produce a more efficient evaluation algorithm was already mentioned by Leroux, but no precise method was described or even outlined. In this work, we show that there are several interesting approaches to consider in this setting.    
For a fixed value of $\ell$, each of our new algorithms will achieve the best known complexity for some range of primes $p$ (relative to the size of $\ell$). Hereafter, we assume that the evaluation is done on an input $j \in \FF_p$ that can be seen as an integer $0 \leq j \leq p-1$. All the stated complexities hold under several plausible heuristic assumptions. 

\begin{enumerate}
    \item \ModularEvaluationBigChar{} is a generic CRT evaluation algorithm built on top of the \OrdersToJBig{} algorithm from \cite{leroux2023computation} which combines \cite[Algorithm 2.2]{sutherland2013evaluation} and \ModularComputation{} from \cite{leroux2023computation}. The complexity is $ O\left(\ell^2 (\ell \log \ell + \log p) \log^{2+\varepsilon} (\ell \log \ell +\log p)\right)$ and the memory requirement is  $O\left( (\ell \log \ell + \log p) \log (\ell \log \ell + \log p) + \ell \log p\right)$. The asymptotic complexity is the same as \cite[Algorithm 1]{sutherland2013evaluation}, but the space requirement is better. This algorithm will achieve the best known space/time complexity in cases where $p$ (i.e., the characteristic) is large. 
    \item \SupersingularEvaluation{} works on supersingular $j$-invariants and has a complexity of $O\left(\ell (\log p^{4 + \varepsilon} + \log \ell^{2 + \varepsilon} \log^{1+ \varepsilon} p) + p^{1/4} \log^{3 + \varepsilon} p \right)$, using $O(\ell \log p \allowbreak + p^{1/4} \log^{1+\varepsilon} p)$ space. Despite a limited range of applications (due to the supersingularity constraint), this algorithm achieves the best known complexity when the prime $p$ is small, and it is the only algorithm that is linear in $\ell$. 
    \item \ModularEvaluationBigLevel{} is a generic CRT evaluation algorithm built on top of \SupersingularEvaluation{}. It has two main steps: a CRT prime collection of complexity $O\left(\ell^2 \log^{2} j  (\log^{2+\varepsilon} (\ell \log j))\right)$, and a CRT computation of complexity $O ( \ell^2 \log j \log^{3+\varepsilon} \ell + \ell^{3/2} \log^{3/2} j \log \ell^{3+\varepsilon} + \log^{2 +\varepsilon} \ell \log^{1+ \varepsilon} p )$. The global space requirement is $O(\ell \log (pj))$. This algorithm achieves the best known complexity of all generic evaluation algorithm if either one of $j$ or $p$ is small (with respect to $\ell$). In the generic case, we will have $j = \Theta(p)$, and the complexity is quadratic in $\log p$, which is worse than the other algorithms. However, note that the quadratic component of the complexity comes from the CRT prime collection, and we note that this step only depends on the input $j$ and can be shared among evaluations of different modular polynomials on the same $j$. This could prove interesting for applications such as point counting. 
\end{enumerate}

Outside of these algorithmic contributions, we also show how to adapt our methods to work for other modular polynomials, such as the one based on the Weber function, by exploiting a modular parametrization of elliptic curves with order $48$ level structure underlying the Weber function. This is very convenient in practice as the sizes of the coefficients of modular polynomials associated to the Weber function are smaller by a large constant, which makes them approximately 1728 times faster to compute than standard modular polynomials of the same level.  

Finally, in \cref{sec: implementation}, we provide an efficient implementation of all our algorithms written in C++ and NTL \cite{ntl}. 
Unfortunately, at their current state, the implementations of our two generic algorithms do not seem to outperform the state-of-the-art implementation from \cite{sutherland2013evaluation} for parameters that are within reasonable computational reach. A more detailed analysis of the concrete performance of our implementation can be found in \cref{sec: implementation results}. 
We emphasise, however, that many of the algorithms we implement are useful outside their application in this paper. We highlight in 
particular the following: 
\begin{enumerate}
    \item An optimised implementation of polynomial interpolation, which outperforms NTL's in-built function for polynomials of large degree.
    \item An implementation of the Deuring correspondence for generic primes, thus providing a low-level optimised implementation 
    of algorithms from \cite{eriksen2023deuring}, which outperforms that implementation by a large constant factor.
    \item An optimised implementation of the computation of the endomorphism ring of a supersingular elliptic curve defined over $\FF_p$ for generic primes. 
\end{enumerate}
The last two routines are especially useful for the cryptanalysis and construction of isogeny-based cryptography.

\vspace{2ex}

\subsubsection*{Acknowledgements.} We thank Drew Sutherland and Sam Frengley for helping us finding the maps that are needed to compute the modular polynomials associated to the Weber function. 

This work was supported in part by the European Research Council (ERC) under the European Union’s
Horizon 2020 research and innovation programme (grant agreement ISOCRYPT -- No.\ 101020788), 
by the Research Council KU Leuven grant C14/24/099, and by CyberSecurity Research Flanders with reference 
number VR20192203. This work was also supported by the European
Research Council under grant No.\ 101116169 (AGATHA CRYPTY)
and by the Deutsche Forschungsgemeinschaft (DFG, German Research Foundation)
under SFB 1119 -- 236615297.

\subsection{Technical overview}
\label{sec: technical overview}

Let $p$, $\ell$ be two distinct primes, and let $j \in \FF_p$. Hereafter, we abuse notation and also write $j$ for the unique integer between $0$ and $p-1$ in the class of $j$ modulo $p$. The goal of all the algorithms introduced in this work is to compute the polynomial $\Phi_\ell(X,j) \in \FF_p[X]$. 

Our work can be seen as an extension of the ideas introduced by Leroux in \cite{leroux2023computation} to the setting of modular polynomial evaluation, using some ideas outlined by Sutherland in \cite{sutherland2013evaluation}. Note that the complexities of our algorithms rely on heuristic assumptions because the complexity of Leroux's algorithm are only proven assuming various heuristics. 

All our algorithms rely on the Deuring correspondence and the \OrdersToJSmall{} and \OrdersToJBig{} algorithms from \cite{leroux2023computation} whose goal is to compute the set of $j$-invariants corresponding (under the Deuring correspondence) to a set of maximal order types given as input. \OrdersToJSmall{} is generally more efficient, unless the set of $j$-invariants to be computed is somewhat close to the entire set of supersingular $j$-invariants, where \OrdersToJBig{} is tailored to be more efficient.
These two algorithms were used by Leroux to build two new efficient algorithms to compute modular polynomials modulo $p$:
\begin{enumerate}
    \item \SupersingularModularComputation{} directly applies \OrdersToJSmall{} (or \OrdersToJBig{}, if it is faster) to compute the necessary $j$-invariants to interpolate $\Phi_\ell(X,Y)$ over $\FF_p$. 
    \item \ModularComputation{} is a CRT algorithm that applies \SupersingularModularComputation{} on a set of small CRT primes $p_i$ before reconstructing the result mod $p$. Here, \SupersingularModularComputation{} is always used with \OrdersToJBig{} because the $p_i$ are small compared to $\ell$.  
\end{enumerate}
The approach by Leroux is amenable to ideas introduced by Sutherland \cite{sutherland2013evaluation} to adapt the BLS modular polynomial computation algorithm to the evaluation setting. 
Sutherland observed that there are essentially two ways to evaluate $\Phi_\ell$ on $j \in \FF_p$.
The first and most direct way is to consider $j$ as an integer, compute $\Phi_\ell(X,j)$ over the integers and finally reduce the result modulo $p$. This leads to \cite[Algorithm 2]{sutherland2013evaluation} which has optimal space complexity, but can be quite inefficient when $\log p$ grows and $j = \Theta(p)$. The main problem with that method is that the powers of $j$ are much bigger than $j$ over $\ZZ$. This increases the size of the coefficients of $\Phi_\ell(X,j) \in \ZZ[X]$ which has a negative impacts on the performances of this approach. 

Sutherland's trick is to realize that we can avoid exponentiating over $\ZZ$ by exponentiating over $\FF_p$ where the powers of $j$ are all in $O(p)$. Thus, he proposes to lift each element $j^k \in \FF_p$ to an integer $\widehat{j_k}$ for $1\leq k \leq \ell+1$, see $\Phi_\ell(X,j)$ as a multivariate polynomial $P_\ell(X,j_1,\dots,j_{\ell+1})$ by replacing each $j^k$ by a variable $j_k$ and evaluate this polynomial on the $\widehat{j_k}$ before reducing the result modulo $p$.
This gives \cite[Algorithm~1]{sutherland2013evaluation}, which obtains the best known complexity but requires more space than the more naive approach. This method has the same complexity as computing the whole modular polynomial, but requires less memory as the full polynomial is never stored in memory. 

Sutherland also presents a ``hybrid'' version of \cite[Algorithm~1]{sutherland2013evaluation} (see \cite[Section 3.4]{sutherland2013evaluation}). This hybrid algorithm optimises the amount of data stored throughout the computation, and so it is slower but uses an optimal amount of space. 
  
Our first algorithm, \ModularEvaluationBigChar{} can be seen as an adaptation of Leroux's \ModularComputation{} with the ideas behind Sutherland's hybrid algorithm. Similarly to Sutherland's idea, we lift powers of $j$ to integers $\widehat{j_1},\ldots,\widehat{j}_{\ell+1}$ before evaluating the polynomial $P_\ell(X,\widehat{j}_1,\dots,\widehat{j}_{\ell+1})$ with a CRT method. The computation of $P_\ell(X,\widehat{j}_1,\dots,\widehat{j}_{\ell+1})$ modulo the small CRT primes $p_i$ mainly relies on \OrdersToJBig{} in a manner similar to \SupersingularModularComputation{}. It turns out that this method requires less storage than the version using ordinary curves, and so we can execute the analogue of Sutherland's hybrid version without any negative impact on the complexity. 

Thus, just as \ModularComputation{} matches the complexity of BLS, \ModularEvaluationBigChar{} essentially matches the complexity of \cite[Algorithm~1]{sutherland2013evaluation} taking $ O\left(\ell^2 (\ell \log \ell + \log p) \log^{2+\varepsilon} (\ell \log \ell +\log p)\right)$, but with a better space complexity of $O\left( (\ell \log \ell + \log p) \log (\ell \log \ell + \log p) + \ell \log p\right)$ that is quasi-linear in $\ell$. 

Our second generic evaluation algorithm \ModularEvaluationBigLevel{} applies the naive approach of computing $\Phi_\ell(X,j)$ over $\ZZ[X]$, but compensates the efficiency loss caused by the large size of the powers of $j$ in $\ZZ$ by noticing that modulo some well-chosen CRT primes $p_i$, the computation of $\Phi_\ell(X,j) \bmod p_i$ can be made much more efficient than the computation of $P_\ell(X,\widehat{j}_1,\dots,\widehat{j}_{\ell+1}) \bmod p_i$ (which is essentially equivalent to the computation of the full $\Phi_\ell(X,Y) \bmod p_i$). The well-chosen primes are those where $j \bmod p_i$ is the $j$-invariant of a supersingular curve and this efficient algorithm is our third algorithmic contribution: \SupersingularEvaluation{}, whose complexity is linear in $\ell$. The main downside of this idea is that the probability that $j \bmod p_i$ is supersingular is in $O(1/\sqrt{p_i})$ which makes the computation of the CRT primes quite costly with a complexity of $O\left(\ell^2 \log^2 j \log^{2+\varepsilon} (\ell \log j)\right)$. In the generic case where $\log j = \theta(\log p)$, this step will be the bottleneck.  

Once the CRT primes have been computed, the rest of the computation takes $O\left( \ell^2 \log j \log \ell^{3+ \varepsilon} + \ell^{3/2} \log^{3/2} j \log^{2+\varepsilon} \ell  + \log^{2 +\varepsilon} \log^{1+ \varepsilon} p
\right)$, and the overall space requirement is $O(\ell \log (pj))$ which is optimal given the size of the output.  

\SupersingularEvaluation{} consists of a rather straightforward application of \OrdersToJSmall{} to compute the $j$-invariants $\ell$-isogenous to the supersingular $j$\nobreakdash-invariant given as input. This gives a linear complexity in $\ell$, but comes at the cost of computing the endomorphism ring of the input curve, which takes $O(p^{1/4} \log p^{3 + \varepsilon})$. The final complexity is $$O\left(\ell( \log \ell^{2 + \varepsilon} \log p^{1 + \varepsilon} + \log p^{4 + \varepsilon}) + p^{1/4} \log p^{3 + \varepsilon}\right).$$ 

We note that \SupersingularEvaluation{} is interesting in its own right, as it achieves the best known complexity to evaluate $\Phi_\ell$ when the prime $p$ is small. However, the constraint of having a supersingular $j$-invariant as input is quite limiting for practical applications.

\vspace{1ex}

\subsubsection*{Other modular functions.} We also show how to adapt our method to compute modular polynomials associated to other modular functions such as the Weber function. It was already suggested by Leroux in \cite{leroux2023computation} that, similarly to the BLS algorithm, the supersingular approach could be adapted to work for other modular functions. However, unlike the prediction by Leroux, who qualified the task as ``not too daunting'', concretely implementing this comes with a few technical obstacles. In particular, the approach taken in BLS cannot be made to work here. Our solution uses a level structure parametrization of Weber invariants and elimination theory to solve the issue. 

\vspace{1ex}

\subsubsection*{Organisation of the article.} The rest of this paper is organised as follows: in \cref{sec: background}, we introduce some background on isogenies, quaternion algebras and the Deuring correspondence. The main technical contributions are introduced in \cref{sec: modular poly evaluation} where we present our new algorithms. Finally, in \cref{sec: implementation}, we give some details on our implementation.

\section{Background material}
\label{sec: background}

Throughout the paper, we consider logarithm in base $2$ that we simply write $\log(\cdot)$, and we use $\llog (\cdot)$ for $\log (\log (\cdot))$. 

\medskip





Background on the main topics covered in this article can be found in:
\begin{enumerate}
  \item Silverman \cite{S09} for elliptic curves and isogenies; 
  \item Voight \cite{voight} for quaternion algebras
  \item the thesis of Leroux \cite{leroux2022quaternion} for the algorithmic aspects of the Deuring correspondence.  
\end{enumerate}

In short, the Deuring corrrespondence makes a link between elliptic curves and isogenies over $\FF_p$ on one side, and orders and ideals of the quaternion algebra ramified at $p$ and $\infty$ on the other side. This result thus provides a way to navigate efficiently the supersingular isogeny graph by using some simple operations over some lattices of dimension $4$. This is what we call the effective Deuring correspondence. 


\textbf{Effective Deuring correspondence.} More precisely, one important algorithm for us is a heuristic polynomial-time algorithm introduced in \cite{EHLMP18} which computes the $j$-invariant corresponding to a maximal order type given as input. It makes use of a smooth isogeny connecting the desired $j$-invariant to a special $j$-invariant (for instance $j=1728$ or $j=0)$. 
This smooth isogeny is found using the Deuring correspondence with the KLPT algorithm \cite{KLPT14}. The smooth isogeny found with KLPT is usually quite large (the heuristic estimates suggest that one should be able to find an isogeny of degree $\tilde{O}(p^3)$ with that algorithm), and so the concrete computation of this isogeny, for a generic prime $p$, can be quite slow. The asymptotic cost of this computation is $O(\log^{4 + \varepsilon} p)$, but the constants hidden in that complexity are dauntingly large for generic primes. Nevertheless, in \cite{eriksen2023deuring} it was shown that this algorithm can be made practical, even for generic primes, by applying several practical improvements to the original algorithm from \cite{EHLMP18} Henceforth, we call this algorithm \OrderToJ{}.


\begin{remark}
  There seems to be a typo in \cite{leroux2023computation} where the cost of this algorithm is estimated to be $O(\log p^{5+\varepsilon})$. However, if we look at \cite[Lemma 4]{GPS17} we see that the cost should be of $O(\log^{4 + \varepsilon} p)$ (after a precomputation of $\log^{6 + \varepsilon} p$ that can be amortized across all isogeny computations). 
\end{remark}


\textbf{Efficient Deuring correspondence for several orders. }
In \cite{leroux2023computation}, Leroux looked at the problem of the effective Deuring correspondence for a set of order types. For a small number of orders (relative to the prime characteristic $p$), the best method seems to be simply to apply \OrderToJ{} to each order in the set. Leroux called this algorithm \OrdersToJSmall, and we will keep that notation throughout the paper. If $S$ is the size of the set of orders to be computed, the complexity of this algorithm is $O(S \log^{4 + \varepsilon} \log p)$. 

However, when the set contains a lot of order types, Leroux introduced a much faster algorithm called \OrdersToJBig{} (the algorithm really becomes interesting for sets of size $O(p)$). The idea of this algorithm is to go through the entire set of supersingular curves through quaternions and collect the desired $j$-invariants along the way in order to minimize the number of isogeny computations required. 

The asymptotic complexity of this algorithm is $O(S \log p^{2 + \varepsilon} + p \log p^{1+\varepsilon})$, and it also has better hidden constants than \OrdersToJSmall{} due to the fact that it involves the computation of isogenies of degree $O(p)$ instead of $O(p^3)$.

\smallskip

\textbf{Application of the Deuring correspondence to the computation of modular polynomials.} The principle that ``quaternion operations" are generally more efficient than ``elliptic curve operations" is at the heart of the recent algorithm \ModularComputation{} of Leroux \cite{leroux2023computation} for computing modular polynomials. Indeed, the main part of this computation is to collect pairs of $\ell$-isogenous $j$-invariants. By working with supersingular $j$-invariants, Leroux showed how to exploit the effective Deuring correspondence to realize that collection with a minimal amount of ``elliptic curve" operations by working mostly with quaternions. 
The idea is that the set of maximal order types corresponding to the $j$-invariants to be collected can be generated purely from quaternion operations. Then, it suffices to apply the algorithms we mention above to collect all the necessary $j$-invariants.

\section{Evaluation of modular polynomials}
\label{sec: modular poly evaluation}

In this section, we introduce our main theoretical contributions, namely the three new algorithms \ModularEvaluationBigChar{}, \SupersingularEvaluation{}, and \ModularEvaluationBigLevel{} (respectively introduced in \cref{sec: eval big char,sec: supersingular eval,sec: eval big level}). In \cref{sec: comparison}, we provide a concrete comparison with several other algorithms from the literature. Finally, in \cref{sec: other modular}, we explain how to adapt our method to evaluate modular polynomials associated with the Weber invariant. 

\subsection{A first CRT approach for big characteristic}
\label{sec: eval big char}

Our first CRT approach is inspired by the ``hybrid'' algorithm described by Sutherland but where \cite[Algorithm 2.2]{sutherland2013evaluation} is replaced with an algorithm \SpecialSupersingularEvaluation{} described below.
This algorithm combines the ideas from \SupersingularModularComputation{} from \cite{leroux2023computation}, with the trick introduced by Sutherland to perform the evaluation with a minimal space requirement. The part of our algorithm that constructs the final output by CRT is identical to the one in \cite[Algorithm 2 ]{sutherland2013evaluation}. We therefore do not describe these steps precisely, and instead refer the reader to \cite[Section 6]{sutherland2011computing} for more details. 

\begin{algorithm}[ht]
    \caption{$\SpecialSupersingularEvaluation{}(p,\ell,x_0,\ldots,x_\ell)$}\label{alg: special supersingular evaluation}
    \begin{algorithmic}[1]

    \Require {A prime $p$, a prime $\ell$ with $\lceil p/12 \rceil +1 < \ell$ and $\ell + 1$ values $x_0,\ldots,x_\ell \in \FF_p$. }
    \Ensure{ $P(Y) = \sum_{i,j} a_{i,j} x_i Y^j $ where $\Phi_\ell(X,Y) = \sum_{i,j} a_{i,j} X^i Y^j$.}
    \STATE \label{step: find all types} Compute the set of maximal order types $\O_{1},\ldots,\O_{m} \subset \QA$ and set $\frakS$ as the list of these maximal order types
    \STATE Compute $J = \OrdersToJBig{}(p, \frakS)$
    \STATE Select a set $\O_{1,0},\ldots,\O_{\ell+2,0}$ in $\frakS$
    \FOR{ $i=1$ to $\ell+2$}
        \STATE Find $j_{i,0}$ as the $j$-invariant in $J$ corresponding to $\O_{i,0}$
        \STATE Compute $I_{i,1},\ldots,I_{i,\ell+1}$, the $\ell+1$ $\O_{i,0}$-ideals of norm $\ell$
        \FOR{$k=1$ to $\ell+1$}
            \STATE Select $j_{i,k}$ as the $j$-invariant in $J$ corresponding to $\O_R(I_{i,k})$
        \ENDFOR
        \STATE $P(j_{i,0},Y) \leftarrow \prod_{k=1}^{\ell+1} (Y-j_{i,k})$
        \STATE Write $P(j_{i,0},X) = \sum_{k=1}^{\ell+1} c_{i,k-1} X^{k-1}$
        \STATE $y_i \leftarrow \sum_{k=1}^{\ell+1} c_{i,k-1} x_{k-1}$
    \ENDFOR

    \STATE Interpolate $P(Y) \bmod p$ as the polynomial of degree $\ell+1$ with $ P(j_{i,0}) =y_i$.
    \RETURN $P(Y)$.  
    \end{algorithmic}
   \end{algorithm}

   \begin{proposition}
    \label{prop: special supersingular evaluation}
    \SpecialSupersingularEvaluation{} is correct, and can be executed in $O(\ell^2 \log^{2+\varepsilon} \ell \log^{1+ \varepsilon} p + p \log^{2+ \varepsilon} p)$, and requires $O(\ell \log p + p\log p)$ space.  
   \end{proposition}
\begin{proof}
    All maximal orders in $\QA$ can be enumerated in $O(p \log p)$ with the techniques described for \SupersingularModularComputation{} from \cite{leroux2023computation}. 
    By \cite[Theorem 1]{leroux2023computation}, the running time of \OrdersToJBig{} is $O(p \log^{2+ \varepsilon} p)$. 
    Thus, all the steps before the main loop can be run in $O(p \log^{2+ \varepsilon} p)$.

    Computing an ideal of norm $\ell$ takes $O(\log (\ell p))$.  
    Thus, the asympotic cost of the loop is dominated by the cost of the polynomial reconstruction at the end which takes $O(\ell \log^{2+\varepsilon} \ell \log^{1+\varepsilon} p)$ with the fast interpolation algorithms.
    Since the loop is repeated $\ell$ times, the final cost is $O(\ell^2 \log^{2+\varepsilon} \ell \log^{1+ \varepsilon} p)$. 
    
    The sets $\frakS$ and $J$ use $O(p \log p)$ space to store, and the space requirement of \OrdersToJBig{} is $O(p \log p)$. 
    The amount of space for each iteration of the loop is $O(\ell \log p)$ to store the $j$-invariants and the polynomials $P(j_{i,0},Y)$. 
    After each iteration of the loop, we need to store one value in $\FF_p$, so the total is $O(\ell \log p)$.
\end{proof}

\begin{remark}
    \label{rmk: frobenius conjugacy} 
    The algorithm outlined as \cref{alg: special supersingular evaluation} is correct up to some small detail: order types correspond to $j$-invariants up to Galois conjugacy. Since we work over $\FF_{p^2}$, this means that there can be up to two distinct $j$-invariants with the same endomorphism ring ($j$ and $j^p$). In practice, this means that our algorithm needs a small modification to know which one of the two is the correct among $j_{i,k}$ and $j_{i,k}^p$. This can be done by modifying \OrdersToJBig{} to store not only couples of order types and $j$-invariants, but also ideals whose corresponding isogeny's codomain is the computed $j$-invariant. Then, we can check if the $j$\nobreakdash-invariant corresponding to $I_{i,k}$ is $j_{i,k}$ or $j_{i,k}^p$ by checking if $I_{i,k}$ is equivalent or not to the ideal stored along with the type of $\O_R(I_{i,k})$.  
\end{remark}

\begin{algorithm}[ht]
    \caption{$\ModularEvaluationBigChar{}(p,j,\ell)$}\label{alg: modular evaluation big char}
    \begin{algorithmic}[1]
    \Require {A prime $p$, $j \in \FF_p$, a prime $\ell$. }
    \Ensure { $\Phi_\ell(j,Y)$.} 
    \STATE Let $(\overline{j_i})$ be the integers in $[0,p-1]$ equal to $j^i \bmod p$ for $1 \leq i \leq \ell+1$
    \STATE Set $B = 2^{6 \ell \log \ell + 18 \ell + \log p + \log (\ell+2)}$.
    \STATE Set $\primesell{} = \lbrace \rbrace$, $P=1$, $q=12(\ell + 2) + 1$. 
    \WHILE { $P < B$ }
        \IF {$q$ is prime}
            \STATE $P \leftarrow q \cdot P$, \enspace $\primesell{} \leftarrow \primesell{} \cup \lbrace q \rbrace$
        \ENDIF
        \STATE $q \leftarrow q+2$
    \ENDWHILE
    \STATE Perform the pre-computations for the explicit CRT mod $p$ using $\primesell{}$
    \FOR {$q \in \primesell$}
    	\STATE Set $j_q$ to be the vector $(\overline{j_i})$ reduced $\bmod q$.
        \STATE $P_q(Y) \leftarrow \SpecialSupersingularEvaluation{}(q,\ell, j_q)$
        \STATE Update the CRT sums for each coefficient of $P_q(Y)$ 
    \ENDFOR 
    \STATE Perform the post-computation for the explicit CRT to obtain $P(Y) \in \FF_p[X]$
   \Return $P(Y)$
    \end{algorithmic}
   \end{algorithm}

\begin{proposition}
    \label{prop: CRT evaluation big characteristic}
    The expected running time of \ModularEvaluationBigChar{} is $ O(\ell^2 (\ell \log \ell + \log p) \log^{2+\varepsilon} (\ell \log \ell +\log p))$ and requires $O( (\ell \log \ell + \log p) \log (\ell \log \ell + \log p) + \ell \log p)$ memory. 
\end{proposition}
\begin{proof}
    With the same reasoning as in the proof of the expected running time of \ModularComputation{} in \cite{leroux2023computation}, 
    we can take $\# \primesell = O(\log B/\llog B)$ and $\max_{q \in \primesell} q = O(\log B)$. 
    We obtain the result by combining $\log B = O(\ell \log \ell + \log p)$ with \cref{prop: special supersingular evaluation}, and the cost of the CRT operations given in \cite{sutherland2013evaluation}. 
\end{proof}

\subsection{The supersingular case}
\label{sec: supersingular eval}

Our algorithm to evaluate the modular polynomial on the $j$-invariant $j$ of a supersingular curve is simple to describe. It mainly consists in the application of the \OrdersToJSmall{} (or \OrdersToJBig{}) algorithms from \cite{leroux2023computation} to collect the $j$-invariants required to reconstruct $\Phi_\ell(j,X)$ from its roots. However, the input of these algorithms are the maximal orders isomorphic to the endomorphism ring of the curve $\ell$-isogenous to $j$. The simplest way to compute these maximal orders is to know the endomorphism ring associated to $j$. Hence, the first building block we need is an  algorithm to compute the endomorphism ring of a supersingular curve defined over $\FF_p$. 

\vspace{1ex}

\subsubsection*{Computing the endomorphism ring of a supersingular elliptic curve over $\FF_p$.}

We introduce an algorithm \EndoRing{} of heuristic complexity $\widetilde{O}(p^{1/4})$. This algorithm relies on a subroutine that computes isogenies between two given supersingular curves over $\FF_p$, which can morally be traced back to \cite[Algorithm 1]{delfs2016computing}. In this paper, it is simply refined to be efficient in practice and allow the computation of the ideal corresponding to the computed isogeny (via the Deuring correspondence). We then simply apply this algorithm between the target curve $E$ and some starting curve $E_0$ of known endomorphism ring. 

The main property behind this algorithm is that the set of supersingular curves $E$ defined over $\FF_p$ with a given $\FF_p$\nobreakdash-endomorphism ring $R:=\End(E)\cap\QQ(\pi)$, where $\pi\colon E\to E$ is the $p$\nobreakdash-power Frobenius endomorphism, admits (very few exceptional cases exempt) a free and transitive group action by $\Cl(R)$. In the supersingular setting, the only two choices are $R=\ZZ[\pi]\cong\ZZ[\sqrt{-p}]$ and $R=\ZZ[\frac{1+\pi}2]\cong\ZZ[\frac{1+\sqrt{-p}}2]$. Thus, the set of supersingular curves over $\FF_p$ with a given $\FF_p$\nobreakdash-endomorphism ring can be obtained by applying the action of a set of elements spanning $\Cl(R)$ on any given starting curve $E_0/\FF_p$ with the correct choice of $\FF_p$\nobreakdash-endomorphism ring ($\ZZ[\pi]$ or $\ZZ[\frac{1+\pi}2]$, corresponding to the surface or floor of the $2$\nobreakdash-isogeny volcano). Heuristically, we can obtain such a set by considering the combination of small powers of $O(\log p)$ ideal classes in $\Cl(R)$. This is formalized as \cref{claim: class group span}.

This heuristic result is convenient in practice because computing a policyclic representation as in \cite[Algorithm 2.2]{sutherland2011computing} requires a number of operations linear in the class number, whereas constructing the set used in the heuristic method can be done in polylogarithmic time.  

\begin{claime} 
    \label{claim: class group span}
    Write $R\in\big\{\ZZ[\pi],\ZZ[\frac{1+\pi}2]\big\}$.
    There exists a constant $C$ such that for any prime $p$, any pair of elements $[\mathfrak{g}_1], [\mathfrak{g}_2]$ in $\Cl(R)$, any set of ideals $\frakl_1,\ldots,\frakl_n$ with pairwise distinct odd prime norms in $R$, and set of exponents $e_1,\ldots,e_n$ such that $\prod_{i=1} (4e_i +1 ) > C p^{1/2} \log p$, there exists $\fraka = \prod_{i=1}^n \frakl_i^{b_i}$ with each $b_i \in [-2e_i,2e_i]$ such that $[\mathfrak{g}_1] = [\fraka] [\mathfrak{g}_2]$ in $\Cl(R)$.  
\end{claime}

\begin{algorithm}[ht]
    \caption{$\EndoRing{}(p,j)$}\label{alg: endo ring}
    \begin{algorithmic}[1]
    \Require {A prime $p$, and the $j$-invariant of a supersingular curve over $\FF_p$, denoted $j$.}
    \Ensure{$\bot$, or maximal order $\O \in \QA$ isomorphic to $\End(E)$ where $j(E) = j$.}
    \STATE Let $\O_0$ be one of the maximal orders of $\QA$ given in \cite[Lemma 2,3,4]{KLPT14}
    \STATE Compute $E_0$ a supersingular curve defined over $\FF_p$ of known endomorphism ring whose endomorphism ring is isomorphic to $\O_0$
    \STATE Compute a set of ideals $\frakl_1,\ldots,\frakl_n$ in $\frakO$ of odd prime norm $\ell_1,\ldots,\ell_n$ in $\frakO$ and exponents $e_1,\ldots,e_n < 10$ such that $ \prod_{i=1}^n (4e_i+1) > C p^{1/2} \log p $
    \STATE Set $J_0 = \lbrace (j_0,\frakO) \rbrace$
    \STATE Set $J = \lbrace (j,\frakO) \rbrace$
    \FOR {$i=1$ to $n$} 
        \FOR { $e \in [-e_i,e_i]$ and $(j_1,\fraka_1),(j_2,\fraka_2) \in J_0 \times J$ }
            \STATE $J_0 = J_0 \cup \lbrace (\frakl_i^b \star j_1, \frakl_i^b \fraka_1) \rbrace $
            \STATE $J = J \cup \lbrace (\frakl_i^b \star j_2, \frakl_i^b \fraka_2) \rbrace$
        \ENDFOR
        \IF{ there is a collision $(j',\fraka), (j',\frakb) \in J_0 \times J$ }
            \RETURN $\O= \O_R( \O_0 \fraka \frakb^{-1} )$
        \ENDIF 
    \ENDFOR
   \Return $\bot$
   \end{algorithmic}
   \end{algorithm}
\begin{proposition}
\label{prop: endo ring}
Assuming \cref{claim: class group span}, \EndoRing{}  always returns a maximal order isomorphic to the endomorphism ring of the input curve. 
 The expected running time of \EndoRing{} is $O(p^{1/4} \log p^{3 + \varepsilon})$, and the expected memory cost is $O(p^{1/4} \log p^{2 + \varepsilon})$. 
 The size of the output is $O(\log p)$.
\end{proposition}
\begin{proof}

    Under $GRH$, the smallest value of $q$ such that $\sqrt{-q}$ is contained in $\QA$ is in $O(\log^2 p)$
     by a result of Ankeny \cite{ankeny1952least}. Thus, with the CM method, a curve $E_0$ can be found using $\polylog{p}$ binary operations. 

    By \cref{claim: class group span} and \cite[Theorem 7]{CLMPR18}, all supersingular $j$-invariants defined over $\FF_p$ will be spanned by $\left( \prod_{i=1}^n \frakl_i^{b_i} \right) * j_0$, for $b_i \in [-2e_i,2e_i]$ with overwhelming probability. Thus, we know there will be a collision in $J$ and $J_0$. 
    
    Under GRH, if we select the norms $\ell_1,\ldots,\ell_n$ of the ideals $\frakl_1,\ldots,\frakl_n$ as the smallest primes that are split in $\ZZ{\sqrt{-n}}$, then we have $\frakl_i = O(\log p^{1+\varepsilon})$ for all $i$, and we have $n = O(\log p)$.

    The computation of isogenies of degree $\ell_i$ can be done in $O(\ell_i M_\PP(\ell_i) M_\ZZ(\log p))$, as Vélu's formulas can be computed in $O(\ell_i)$ operations over the field of definition of the kernel. The degree of the field of definition of the $\ell_i$-torsion over $\FF_p$ is $O(\ell_i)$
 (see \cite[Lemma 2]{leroux2023computation} for instance). And so arithmetic operations over that field can be performed in $O (M_\PP(\ell_i) M_\ZZ(\log p))$. 

    The kernel of the isogenies realizing the action of $\frakl_i$ can be computed as the eigenvalue of the Frobenius morphism. These points can be computed by evaluating the Frobenius on a basis of the $\ell_i$-torsion. 
    The overall cost of this operation is $O(\sqrt{\ell_i} M_\PP(\ell_i) M_\ZZ(\log p) )$. 

    With $\ell_i = O(\log p)$, we get the final bound of  $O (\log p^{3+\varepsilon})$ on the cost of each group action computation. 

    Under GRH, we have $h(\frakO) = O (\sqrt{p} \, \llog(p))$, and so, by the birthday paradox, the number of $j$-invariants that needs to be computed in the sets $J$ and $J_0$ before a collision can be found is in $O(p^{1/4} \log^{\varepsilon} p )$.
    Thus, the expected running time of \EndoRing{} is $O(p^{1/4} \log^{3+\varepsilon} p)$. 
    
    The order $\O_0$ can be computed with coefficients over $\QQ$ of size $O(\log p)$ (see \cite{KLPT14}, for instance). 
    The computation of $\O$ can be done in $O(\log p)$ as the coefficients of the ideal $\O \fraka\frakb^{-1}$ are in $O(p)$.  The maximal order $\O$ can be given by 16 coefficients over $\QQ$. Since $\O$ it is the right order of an ideal of norm in $O(p)$, the size of the coefficients is $O(\log p)$.  
    The $j$-invariants and the ideals take $O(\log p)$ to store so the total memory cost is $O(p^{1/4} \log^{1+\varepsilon} p)$. 

    The output is the right order of the ideal $\O_0 \fraka \frakb^{-1}$. This ideal has norm $O(\log p)$ and so the coefficients of its right order over the basis $1,i,j,k$ of $\QA$ can be upper-bounded by $O(\log p)$. Indeed, it can be easily verified that $\O n(\fraka \frakb) \subset \O_0 $, and so we can express a basis of $\O$ as elements of $\O_0$ divided by $n(\fraka \frakb)$. This gives the desired upper-bound on the size of $\O$.    
\end{proof}

\vspace{1ex}

\subsubsection*{The supersingular evaluation algorithm.} Computing $\Phi_\ell(j,X)$ mainly consists in computing the $j$-invariants that are $\ell$-isogenous to $j$. When $j$ is supersingular, we run this operation without any $\ell$-isogeny computation by computing the endomorphism ring of $j$ and using the Deuring correspondence to compute the $\ell$-isogenous $j$-invariants. 

\begin{algorithm}[ht]
    \caption{$\SupersingularEvaluation{}(p,j,\ell)$}\label{alg: supersingular evaluation}
    \begin{algorithmic}[1]
    \Require {A prime $p$, the $j$-invariant of a supersingular curve over $\FF_p$, a prime $\ell$. }
    \Ensure { $\Phi_\ell(j,Y)$.}
    \STATE Compute $E$ the curve whose $j$-invariant is $j$
    \STATE Compute $\O = \EndoRing{}(p,j)$
    \STATE Compute $I_1,\cdots,I_{\ell+1}$ the $\ell+1$ $\O$-ideals of norm $\ell$
    \STATE Set $\frakS = \lbrace \O_R(I_i) | i \in \lbrace 1,\ldots,\ell+1 \rbrace\rbrace$
    \STATE\label{step: modular j comput} Compute $J = \OrdersToJSmall{}(p,\frakS)$ (or use \OrdersToJBig{}$(p,\frakS)$ if this is more efficient)
   \Return $\prod_{j_i \in J} (Y-j_i)$
    \end{algorithmic}
   \end{algorithm}

   \begin{proposition}
    \label{prop: supersingular evaluation} 
    \SupersingularEvaluation{} can be executed in 
    $$O\left(\ell (\log \ell^{2 + \varepsilon} \log p^{1 + \varepsilon} + \log p^{2 + \varepsilon} ) + p \log p^{1 + \varepsilon}\right)$$
    binary operations and requires $O( (\ell + p) \log p)$ space when \OrdersToJBig{} is used in Step~\ref{step: modular j comput} (assuming GRH,\cite[Claim 1]{leroux2023computation} and \cref{claim: class group span}). When using \OrdersToJSmall{}, it can be executed in
    $$O\left(\ell (\log p^{4 + \varepsilon} + \log \ell^{2 + \varepsilon} \log^{1+ \varepsilon} p) + p^{1/4} \log^{3 + \varepsilon} p \right)$$
    binary operations, requiring $O(\ell \log p + p^{1/4} \log^{1+\varepsilon} p )$ space (under GRH, \cref{claim: class group span} and the heuristics from \cite{KLPT14}).  
\end{proposition}
\begin{proof}

    The correctness follows directly from the correctness of the sub-algorithms. 

    The computation of an $\O$-ideal of norm $\ell$ can be done in $O(C + \log \ell)$ where $C$ is a bound on the size of the coefficients of $\O$. Thus, by \cref{prop: endo ring}, the complexity of computing the $O(\ell)$ ideals is $O( \ell( \log \ell + \log p) )$. The computation of the rights orders has the same complexity.
    
    Then, the result follows from \cite[Theorem 1]{leroux2023computation} for \OrdersToJBig{}, the complexity of $O( \ell \log^{4 + \varepsilon} p)$ for \OrdersToJSmall{} for a set of size $O(\ell)$, \cref{prop: supersingular evaluation} for \EndoRing{} and the complexity of the algorithm to compute a polynomial from its roots by building a product tree. 

    The memory requirements follow from the same results. 
\end{proof}

\subsection{A second CRT algorithm for big level}
\label{sec: eval big level}

We obtain our second generic algorithm to evaluate modular polynomials by applying \SupersingularEvaluation{} on a set of well-chosen small primes and then reconstructing the desired result using the CRT method. The only constraint on the choice of the CRT primes is that we need to find primes $p_i$ where the reduction modulo $p_i$ of the $j$-invariant to be evaluated is supersingular.

\begin{algorithm}[ht]
    \caption{$\ModularEvaluationBigLevel{}(p,j,\ell)$}\label{alg: modular evaluation big level}
    \begin{algorithmic}[1]
    \Require {A prime $p$, $j \in \FF_p$, a prime $\ell$. }
    \Ensure { $\Phi_\ell(j,Y) \in \FF_p[X]$.} 
    \STATE Let $\overline{j}$ be the integer in $[0,p-1]$ equal to $j \mod p$
    \STATE Set $B = 2^{6 \ell \log \ell + 18 \ell + (\ell+1) \log j + \log (\ell+2)} $
      
    \STATE $\primesell(j) \leftarrow \lbrace \rbrace$, $P \leftarrow 1$
    \STATE $\Delta \leftarrow \lceil \log B \rceil$
    \STATE Compute $\primesdelta$ as the set of primes smaller than $\Delta$ with Eratosthenes sieve
    \STATE $q_\delta \leftarrow \max \primesdelta$, $n \leftarrow q_\Delta +2$
    \WHILE { $P < B$ }
        \STATE $S \leftarrow [n,n+1,\ldots,n+\Delta]$   
        \STATE Remove all multiples of the elements of $\primesdelta$ from $S$
        \FOR{$q \in S$}
            \STATE Let $j_q = \overline{j} \mod q$ and $E_q$ be the elliptic curve over $\FF_q$ of $j$-invariant equal to $j_q$
            \IF{$E_q$ is supersingular}
                \STATE $P \leftarrow q \cdot P$, \enspace $\primesell{}(j) \leftarrow \primesell{}(j) \cup \lbrace q \rbrace$
            \ENDIF
        \ENDFOR 
        \STATE $n \leftarrow n + \Delta+1$
        \IF {$q_\Delta^2  \leq n$} 
            \STATE $S = [q_\Delta+2,q_\Delta+3,\ldots, 2 q_\Delta+1]$
            \STATE Remove all multiples of the elements of $\primesdelta$ from $S$
            \STATE $\primesdelta \leftarrow \primesdelta \bigcup S$
            \STATE $q_\Delta \leftarrow \max \primesdelta$
        \ENDIF
    \ENDWHILE
    \STATE Perform the pre-computations for the explicit CRT mod $p$ using $\primesell{}$
    \FOR {$q \in \primesell$}
        \STATE $P_q(Y) \leftarrow \SupersingularEvaluation{}(q,j_q,\ell)$ 
        \STATE Update the CRT sums for each coefficient of $P_q(Y)$ 
    \ENDFOR 
    \STATE Perform the post-computation for the explicit CRT to obtain $P(Y) \in \FF_p[X]$ 
   \RETURN $P(Y)$
    \end{algorithmic}
   \end{algorithm}

\begin{proposition}
    \label{prop: modular evaluation}
    \ModularEvaluationBigLevel{} is correct and, under the Lang--Trotter Conjecture~\cite{lang2006frobenius}, \cref{claim: class group span}, and the heuristic from \cite{KLPT14}, the computation of $\primesell(j)$ can be done in
    \begin{equation*}
        O \left( \ell^2 \log^{2} j (\log^{1+\varepsilon} (\ell \log j)) \right)
    \end{equation*} 
    and the rest of the computation in 
    \begin{equation*}
        O \left( \ell^2 \log j \log^{3+\varepsilon} \ell + \ell^{3/2} \log^{3/2} j \log \ell^{3+\varepsilon} + \ell \log^{2 +\varepsilon} \log^{1+ \varepsilon} p  \right). 
    \end{equation*} 
    The memory complexity is $O( \ell \log (pj))$ for each step. 

\end{proposition}

\begin{proof}
    According to the Lang--Trotter conjecture, there are $O(\sqrt{x}/\log x)$ primes $q$ smaller than $x$ such that the reduction of any given $j \in \QQ$ modulo $q$ is supersingular. 
    To ensure that their product is bigger than $B$, we can take $x = O(\log^2 B)$, and there will be $O(\log B/\llog B)$ different primes.  
    
    With $\log B = O(\ell \log j) $, we obtain $\# \primesell{}(j) = O( \ell \log j / \log (\ell \log j))$ and so $\max_{q \in \primesell{}(j)} q = O( \ell^{2} \log j^{2})$. The set $\primesell(j)$ takes $O(\ell \log j)$ memory to store. 
    
    The algorithm we propose to use to compute $\primesell(j)$ is a simple variation of the segmented sieve of Bays and Hudson \cite{bays1977segmented} that can enumerate through all primes smaller than a given $x$ in $O(x \llog x)$ complexity while using $O(\sqrt{x})$ space when $\Delta = O(\sqrt{x})$, and combine it with supersingularity tests on the fly (to avoid storing a list of all the primes smaller than $x$). Supersingularity testing over $\FF_q$ can be performed in $O(\log^{2+\varepsilon} q)$ \cite{sutherland2012identifying}, and it is only performed on prime numbers. Thus, we end up with a total complexity of $O \left( \ell^2 \log^{2} j (\log^{1+\varepsilon} (\ell \log j)) \right)$ and a space complexity of $O(\ell \log j)$ to compute $\primesell(j)$.

    For each $q$, by \cref{prop: special supersingular evaluation} and the fact that $q = O(\ell^2 \log j)$, the complexity of the execution of \SupersingularEvaluation{} with \OrdersToJSmall{} (which will be faster than using \OrdersToJSmall because $p = 
    \Theta(\ell^2 \log^2 j)$) is $$O\left(  \ell (\log^{4 + \varepsilon} (\ell \log j) + \log^{2 + \varepsilon} \ell \log^{1+ \varepsilon} (\ell \log j)) + \ell^{1/2} \log j^{1/2} \log (\ell \log j)^{3 + \varepsilon} \right).$$
    We deduce that the global cost of all the executions of \SupersingularEvaluation{} is
    \begin{equation*}
        O \left(\ell^2 \log j \log^{3+\varepsilon} (\ell \log j)  + \ell^{3/2} \log^{3/2} j \log^{2+\varepsilon} (\ell \log p) \right).
    \end{equation*}
    Finally, the cost of the CRT computation is $O(\log^{2 +\varepsilon} \log^{1+ \varepsilon} p)$ as shown in \cite{broker2012modular}.
    This concludes the proof of the final complexity result. 

    The prime search can be done with $O(\ell \log j)$ memory (to store the set $\primesell$). 
    Each execution of \SupersingularEvaluation{} requires a memory of $$O(\ell \log( \ell \log j ) ) + \ell^{1/2} \log^{1/2} j \log^{1+\varepsilon} \log (\ell \log j).$$  
    The only thing to store between each CRT prime computation are the CRT data whose size is $O(\ell \log p)$. 
    This proves the result on the memory requirement. 
\end{proof}

\begin{remark}
    We note that unlike most other CRT algorithms, in the generic case where $j = \Theta(p)$, the asymptotically dominant step of \ModularEvaluationBigLevel{} is the selection of the primes. 
\end{remark}

\vspace{1ex}

\subsubsection*{Sharing the computation of the CRT primes.}
The set of primes $\primesell(j)$ only depends on the value of $j$. Thus, if we want to evaluate modular polynomials of different levels at the same $j$-invariant (like in the SEA algorithm), the cost of the selection of the primes can be done once and for all for the biggest levels, and then the computation for each level can simply use a subset of these primes. 
This means that for enough levels, the bottleneck will not necessarily be the computation of the primes anymore. 
In particular, in point counting for elliptic curves over a finite field of prime characteristic of bitsize $n$ where one needs to evaluate $j$ on $O(n)$ modular polynomials of level $O(n)$, the cost of the prime selection is amortized over the different levels.

\subsection{Comparison between the existing methods}
\label{sec: comparison}

The best algorithms from the literature are Sutherland's algorithms from \cite{sutherland2013evaluation} and Robert's algorithm from \cite{robert2022some}. We are going to compare our new algorithms with them. 
To establish what is the best asymptotic algorithm, we fix a value of $\ell$ and vary the value of $p$ (and $j$ when it is relevant). We will see that the best algorithm will change as $p$ grows. 

In \cite{sutherland2013evaluation}, Sutherland introduces three different CRT algorithms to compute modular polynomials. The main one has a complexity of $$O\left(\ell^2 (\ell \log \ell + \log p) \log^{2+\varepsilon} (\ell \log \ell +\log p)\right),$$ with a $O(\ell \log p + \ell^2 \log ( \ell \log \ell + \log p ))$ space requirement. The complexity is essentially the same as the cost of computing the entire modular polynomial with a CRT algorithm, but the memory requirement is smaller. Despite this, the space complexity is quadratic in $\ell$, which is not optimal since the size of the output is $O(\ell \log p)$. 
The second algorithm due to Sutherland achieves a $O(\ell^{3+\varepsilon} \log p)$ complexity with an optimal space requirement.
The third is a hybrid version of the first two and has a complexity of $O(\ell^3 \log^{6+\varepsilon} \ell)$. It also requires an optimal amount of space. 

More recently in \cite{robert2022some}, Robert briefly outlined an algorithm with time complexity equal to $O(\ell^2 \log^{\alpha + \varepsilon} \ell \log p)$ where $\alpha$ is at best equal to $1 + 2u$ (according to \cite[Proposition 3.1]{robert2022some}) for $u \in \lbrace 1,2,4 \rbrace$ such that $\ell$ minus a product of small primes powers is equal to a sum of $u$ squares. Heuristically, by adjusting the product of small primes, one can expect to find a case where $u=2$. Thus, at best, the complexity of Robert's algorithm should be $O(\ell^2 \log^{5+\varepsilon} \ell \log p )$. However, note that Robert's paper only gives an outline of the algorithm and of its complexity analysis. Without clear statements, it is hard to assess the real complexity of the algorithm. 
Indeed, from the statements in \cite{robert2022some}, it is not completely clear to us if his asymptotic estimate can be used in the generic case. If not, then the correct value of $\alpha$ is actually $3+2u = 7$.    
The space complexity of Robert's algorithm is not clearly stated either, but it does not seem to be optimal. 

Thus, for small values of $p$ or $j$, our algorithm \ModularEvaluationBigLevel{} appears to have the best known asymptotic complexity. In particular, if $j=O(p^{\varepsilon})$ for $\varepsilon < 1$, then our algorithm has the best complexity of all existing algorithms.

If $j = \Theta(p)$, as $p$ grows, the quadratic factor in $\log j$ in our complexity will mean that \ModularEvaluationBigLevel{} will be outperformed by Robert's algorithm. The asymptotic complexity tells us that the breaking point should occur when $\log p \approx \log \ell^{1+2u}$. 
Although, note that Robert's algorithm probably requires more memory, therefore there may be some practical case where our algorithm would still outperform Robert's. Also note that if we remove the cost of the CRT prime computation (that can be shared among the computation for several levels as already argued), then we expect our algorithm to outperform Robert's for a much larger range of primes (up to $ \log p \approx \ell $).  

When $\ell = O(\log p)$, the best complexity will be obtained by Sutherland's first CRT algorithm and our \ModularEvaluationBigChar{} algorithm, but \ModularEvaluationBigChar{} has a better space complexity than Sutherland's algorithm (which is quadratic in $\ell$). 



Note that the proof of the complexity of Robert's algorithm does not require any assumptions, whereas our method and Sutherland's are only heuristic.   

In \cref{sec: implementation results}, we will compare the practical performance of our C++ implementation with the one from Sutherland.

\subsection{Other modular functions}
\label{sec: other modular}

Modular polynomials can be generalized to modular functions other than the $j$-function. Considering other kinds of modular polynomials is a well-known trick to make computations more efficient (see for instance \cite{broker2012modular}). Indeed, these alternate modular polynomials may have smaller coefficients. One of the most interesting examples appears to be the modular polynomial associated to the Weber function $\mathfrak{f}$. Over $\CC(j)$, this function is a root of the polynomial
\begin{equation}
    \label{eq: minimal weber polynomial}
    \Psi(X,j) = (X^{24} - 16)^3 - j X^{24}
\end{equation} 
and the logarithmic height bound of the associated modular polynomial $ \Phi_\ell^\mathfrak{\ell}$ is expected to be 72 times smaller than the one of $\Phi_\ell$. Moreover, the coefficient of $X^a Y^b$ is non-zero if and only if $\ell a + b = \ell+1 \mod 24$. 
These two facts imply that the complexity of computing $\Phi_\ell^\mathfrak{f}$ should be roughly $72 \times 42 = 1728$ faster than $\Phi_\ell$. By using the simple algebraic relation between $j$ and $\mathfrak{f}$ given by \cref{eq: minimal weber polynomial}, it quickly becomes more efficient to use $\Phi_\ell^\mathfrak{f}$ rather than $\Phi_\ell$.   

Over $\FF_q$, every $j$-invariant is associated to several $\mathfrak{f}$-invariants that are the roots of $\Psi(X,j)$ given in \cref{eq: minimal weber polynomial}. In fact, this polynomial comes from the covering map of $X(1) \cong \PP^1$ by $X_H$, a modular curve of level $48$ isomorphic to the map labelled 48.72.0.d.1 in the (Beta version of the) LMFDB \cite{lmfdb}.  The polynomial
$\Phi_\ell^\mathfrak{f}$ can be constructed from $\mathfrak{f}$-invariants in a manner analogous to $\Phi_\ell$ from $j$-invariants. When $\ell$ is coprime to $48$, it makes sense to talk about $\ell$-isogenous $\mathfrak{f}$-invariants, and evaluating $\Phi_\ell^\mathfrak{f}(f_0,X)$ can be done by computing all the $\ell$-isogenous $\mathfrak{f}$-invariants. 

This idea was described in BLS \cite{broker2012modular}, and Leroux already stated in \cite{leroux2023computation} that the BLS method could be extended to this setting analogously. The goal of this section is to explain how this can be realized concretely. We find that the BLS method cannot be directly applied, and instead new ideas are needed. The main obstacle of using $\frakf$-invariants instead of $j$-invariants is the multiplicity of the polynomial $\Psi$, which implies that there are several $\frakf$-invariants corresponding to the same $j$\nobreakdash-invariant (in the worst case, 72). Therefore, to find the $\ell
$-isogenous $\frakf$-invariants, the usual method of projecting from $X_H$ down to $X(1)$, finding an $\ell$-isogenous $j$\nobreakdash-invariant, and then lifting again to the correct $\frakf$-invariant does not work. Indeed, it is hard to know which lift of $j$ to consider. 

\vspace{1ex}

\subsubsection*{The method used in \cite{broker2012modular}.} Bröker, Lauter, and Sutherland overcome this obstacle by working purely over $X_H$. Indeed, by using the class group action of a quadratic imaginary order of conductor $\ell$ on $X(1)$, they enumerate the set of neighbours in the $\ell$-isogeny graph. 
This action can then be extended to $X_H$. More explicitly, by using the modular polynomial of level $\ell_i$, one can compute the action of ideals of small norm $\ell_i$. When working over $X_H$, we can simply use the modular polynomial $\Phi^\frakf_{\ell_i}$ instead, and the roots will directly give the $\ell_i$-isogenous $\frakf$-invariants. We furthermore remark that this method works as the BLS method only needs to consider points of $X_H(\FF_p)$.

To adapt this idea to our setting of supersingular curves defined over $\FF_{p^2}$, the main problem is that the required $j$-invariants are not computed with the help of modular polynomials, but rather as the codomain of certain isogenies. As there are no efficient isogeny formul\ae{} for isogenies of degree $\ell$ between elements of $X_H(\FF_{p^2})$, we cannot directly derive the $\ell$-isogenous $\frakf$-invariants. Our idea is to use a different interpretation of the curve $X_H$. Indeed, modular curves are known to parametrize elliptic curves enriched with level structure. Thus, if we have an explicit way to associate the elements of $X_H$ with curves and associated level structure of order $48$ (meaning that we can compute the value of the $\frakf$-invariant only from the curve and the level structure), then it suffices to push the level structure through the isogenies of degree $\ell$ to be able to recover the $\ell$-isogenous $\frakf$-invariant. This will work for $\ell$ coprime to $48$, hence for every prime $\geq 5$. 

\vspace{1ex}

\subsubsection*{The parametrization associated to $X_H$.}  The level of $X_H$ is $48 = 3 \times 16$, so we instead consider the level structure of order $3$ and $16$ separately. 

Let us first consider the level $3$ part. 
For this, we look at \cref{eq: minimal weber polynomial} and observe that if $f$ is a root of $\Psi(X,j)$ for a given $j$, then $f^8$ is a root of $\Psi'(X,j^{1/3}) = (X^3 - 16) - Xj^{1/3}$. This is convenient because the $j$-function is well-known to be the cube of another modular function of level $3$, often denoted by $\gamma_2$. In terms of modular curves, this can be interpreted as the cover map 
\begin{align*}
    X_{ns}^+ (3) \rightarrow X(1), \ t \mapsto t^3.
\end{align*}
There is a classical formula to compute the three possible $\gamma_2$-invariants above a given $j$-invariant of a curve $E$ and the $x$-coordinates of $E[3]$ (see for instance in \cite[Section 6.6, Example 1]{brokerthesis}). If $E$ is the curve $y^2 = x^3 + Ax +B$ and $x_1,x_2,x_3,x_4$ are the $x$-coordinates of the non-trivial points of $E[3]$, then the roots of $X^3 - j(E)$ are given by 
\[
    \frac{-48 A}{2A - 3 (x_1 x_2 + x_3 x_4)},
    \quad
    \frac{-48 A}{2A - 3 (x_1 x_3 + x_2 x_4)},
    \quad
    \frac{-48 A}{2A - 3 (x_1 x_4 + x_3 x_2)}
    .
\]
In summary, given an ordering for the $x$-coordinates of the non-trivial points of $E[3]$, we can compute the three possible $\gamma_2$-invariants associated to $E$.
This covers the part of level $3$. 

We now consider the level $16$ part. Here, we look for the solutions of $\Psi''(X,j) = (X^8 - 16)^3 - j X^8$, which will correspond to $\frakf^3$. 
The map $t \mapsto \frac{(t^8-16)^3}{t^8}$ corresponds to the cover of $X(1)$ by the modular curve $X_{H'}$ of level 16 labelled as 16.24.0.p.1 in the new beta version of LMFDB. 
To the best of our knowledge, there are no known formulas to recover the $\frakf^3$-invariant from a level structure of order $16$. Such formulas certainly exist, but it is unclear to us how can they be computed efficiently. Instead, we propose to use the much better studied modular curve $X_s(16)$ above $X_{H'}$. In particular, $X_s(16)$ parametrizes curves together with two subgroups of order $16$ which intersect trivially. In this way, given a curve $E(\FF_q)$ and two subgroups $G_1,G_2 \subset E[16](\FF_q)$, we can recover the corresponding element of $X_s(16)(\FF_q)$ by looking at the common preimage of $j(E)$ through the two maps $X_s(16) \rightarrow X(1)$ factoring through the two covers 
\begin{align*}
    \psi_1 \colon X_s(16) &\rightarrow X_0(16), \quad E,G_1,G_2 \mapsto E,G_1, \\ 
    \psi_2 \colon X_s(16) &\rightarrow X_0(16), \quad E,G_1,G_2 \mapsto E,G_2;
\end{align*}
see \cref{fig: modular covers}. 

\begin{remark}
    \label{rmk: weber problem}
The method that we outlined above works most of the time, but there are some problems when there exist 16-isogenies between the same pair of curves (as the map $X_0(16) \rightarrow X(1) \times X(1)$ is not injective anymore and so our GCD polynomial will have several roots). In that case, it is still possible to identify the element of $X_0(16)$ corresponding to the given subgroup of order 16 by decomposing the 16-isogeny associated to this subgroup as $4$ isogenies of degree $2$, finding the $4$ $X_0(2)$ points corresponding to each of these $2$-isogenies, and finally identifying the correct root of the GCD with the different existing maps from $X_0(16)$ to $X_0(2)$.    
\end{remark}

\begin{figure}[ht!]
    \begin{center}
    \begin{tikzpicture}
        
        \node (Xs16) at (2,2) {$X_s(16)$};
        \node (X0161) at (1,1) {$X_0(16)$};
        \node (X0162) at (3,1) {$X_0(16)$};
        \node (X1) at (2,-1) {$X(1)$};
        \node (XH) at (-1,0.75) {$X_{H}'$};
        \draw [->] (Xs16) -- (X0161);
        \draw [->] (Xs16) -- (X0162);
        \draw [<-] (X1) -- (X0161);
        \draw [<-] (X1) -- (X0162); 
        \draw [->] (Xs16) to [out=180,in=60] (XH);
        \draw [->] (XH) to [out=270,in=180] (X1);
    \end{tikzpicture}
    \caption{Modular covers of $X(1)$ by $X_s(16)$ \label{fig: modular covers}}
\end{center}
\end{figure}
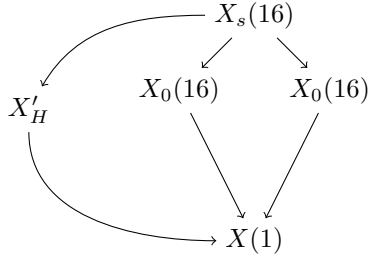

We do this as follows:
\begin{enumerate}
    \item Compute the points of $X_0(16)$ lying above the two pairs $(j, j_1)$ and $(j, j_2)$ of $j$-invariants, i.e., for $k = 1,2$, compute $x_k$ such that $\rho(x_k) = (j, j_k)$ where $\rho : X_0(16) \rightarrow X(1) \times X(1)$ is given by 
    \[ \rho(X) \colonequals \left( \frac{G_0(16)(X)}{H_0(16)(X)}, \frac{I_0(16)(X)}{J_0(16)(X)} \right),\]
    for $G_0(16)(X), H_0(16)(X), I_0(16)(X), J_0(16)(X)\in \ZZ[X]$. If it fails because there are multiple solutions, use the method outlined in \cref{rmk: weber problem}
    \item Recover the point $(x,y)$ on $X_s(16)$ above the two points of $X_0(16)$ by finding a solution to the system of equations consisting
    of
    \[
        \mbox{}\qquad\qquad\qquad
        F_s(16)(X,Y) = 0
        \qquad
        \frac{G_s(16)(X)}{H_s(16)(X,Y)} = (x_1)
        \qquad
        \frac{I_s(16)(X)}{J_s(16)(X,Y)} = (x_2),
    \]
    for $F_s, G_s, H_s, I_s, J_s \in \ZZ[X,Y]$. Where $F$ is the equation of $X_s(16)$, and $\frac{G_s}{H_s}, \frac{I_s}{J_s}$ give the two projection maps $X_s(16) \rightarrow X_0(16).$

    \item Project this point $(x,y)$ on $X_s(16)$ to $X_H'$ via $\phi\colon X_s(16) \rightarrow X'_H$, defined as 
        \[
            \mbox{}\qquad\qquad
            \phi(X,Y) \colonequals \frac{-16X^6Y^{12} - X^6Y^4 + 64X^2Y^{16} + 20X^2Y^8 + X^2}{ 32Y^{15} + 4Y^7}.
        \]
        Then, $\phi(x,y)= \frakf^3$.
\end{enumerate}
For convenience, we recall the formulas for all the maps involved in the computation described above in \cref{appendix:mod-curve-formulae}. 
Some of these maps can be found on the LMFDB, and the other maps were computed using MAGMA \cite{magma}. 

To develop an efficient algorithm that follows the procedure above, we use the method detailed in \cite[\S 5]{splitsearcher}. We first define polynomials
\begin{align*}
    f_1(X,Y) &\colonequals F_s(16)(X,Y), \\ 
    f_2(X,Y) &\colonequals G_s(16)(X,Y) - \alpha_1(j, j_1)H_s(16)(X,Y), \\
    f_3(X,Y) &\colonequals I_s(16)(X,Y) - \alpha_2(j, j_2)J_s(16)(X,Y),
\end{align*}
where, for $k = 1,2$, $\alpha_k \in \ZZ[j, j_k]$ is defined as 
\begin{align*}
    \alpha_k(j, j_k) \colonequals \gcd \bigl(G_0(16)(X) - j H_0(16)(X), I_0(16)(X) - j_k J_0(16)(X)\bigr).
\end{align*}
For $i,j \in \{1,2,3\}$, we also define polynomials 
$$R_{i,j}(Y) \colonequals {\rm{res}}_{X}(f_i(X,Y), f_j(X,Y)) \in \ZZ[j, j_{1}, j_{2}][Y].$$
By the elimination property of the resultant (e.g., see~\cite[\S 3.6, Lemma 1]{MR2290010}), the specialisations $(R_{i,j})_{[E, G_1, G_2]}(Y)$ given by evaluating the coefficients of each $R_{i,j}$ at $j \colonequals j(E)$, $j_1 \colonequals j(E/G_1) $ and $j_2 \colonequals j(E/G_2)$, vanish at the $Y$-coordinate of any common solution to the specialised polynomials $(f_j)_{[E, G_1, G_2]}(X,Y)$.

However, these resultants (generically) have factors which correspond to spurious (unwanted) solutions. Therefore, we instead consider polynomials where these spurious solutions have been removed, namely:
\begin{align*} 
    P_{1,2}(Y) &\colonequals \frac{16}{Y^{108}(16Y^8+1)^7}\cdot R_{1,2}(Y), \text{ and } \\ 
    P_{2,3}(Y) &\colonequals \frac{16}{Y^{77}(16Y^8+1)^7}\cdot R_{2,3}(Y).
\end{align*}
Note that now $P_{1,2}$ and $P_{2,3}$ are coprime. 

If there exist $x,y \in \Fp$ such that $(f_i)_{[E, G_1, G_2]}(x,y) = 0$ for each $i = 1,2,3$ then the degree of $$g(Y) \colonequals \gcd\bigl( (P_{1,2})_{[E, G_1, G_2]}(Y), (P_{2,3})_{[E, G_1, G_2]}(Y)\bigr)$$
is $1$. 
Conversely, if $y \in \Fp$ is a root of $g(Y)$, then there exist $x,x' \in \Fp$ such that 
\begin{align*} 
    (f_1)_{[E, G_1, G_2]}(x,y) &= (f_2)_{[E, G_1, G_2]}(x,y) = 0, \text{ and } \\
    (f_2)_{[E, G_1, G_2]}(x',y) &= (f_3)_{[E, G_1, G_2]}(x',y) = 0.
\end{align*}
We assume $x = x'$ (otherwise, we throw an error and restart the procedure; note this will only happen with negligible probability). 

Therefore, from a $\gcd$ computation, we can extract $y \in \Fp$, and then recover $x$ via a square root computation. Then, $(x,y)$ is a point on $X_s(16)$, and we have $\frakf^3 = \phi(x,y)$, where $\phi : X_s(16) \rightarrow X_H'$. 

In practice, rather than computing the polynomials $P_{1,2}$ and $P_{2,3}$ on the fly, we precompute and store them. Then, given $E, G_1, G_2$ defined over $\Fpbar$, we evaluate the polynomials at corresponding $j, j_1, j_2$, reduce them modulo $p$, and compute their $\gcd$. 
We summarise the discussion above in \cref{alg:getweber}. Here,
\textsf{Roots} denotes a function that returns the roots of a polynomial over a finite field.

\begin{algorithm}[ht]
    \caption{$\GetWeberCube{}(p, j, j_1, j_2)$}\label{alg:getweber}
    \begin{algorithmic}[1]
    \Require{A prime $p$, and $j$-invariants $j, j_1, j_2$ of $E, E/G_1, E/G_2$ respectively.}
    \Ensure{The cube $\frakf^3$ of the $\frakf$-invariant associated to $E$.}
    \STATE Evaluate the coefficients of $P_{1,2}(Y)$ at $j,j_1, j_2$ and reduce modulo $p$ to obtain $h_1(Y) \in \Fp[Y]$
    \STATE Evaluate the coefficients of $P_{2,3}$ at $j,j_1, j_2$ and reduce modulo $p$ to obtain $h_2(Y) \in \Fp[Y]$
    \STATE $c_0Y + c_1 \gets \gcd(h_1, h_2)$
    \STATE $y \gets -c_1\cdot c_0^{-1}$
    \STATE ${\tt rts} \gets \textsf{Roots}(f_1(X,y))$
    \FOR{$r$ in {\tt rts}}
        \IF{$f_3(r,y) = 0$}
            \RETURN $\phi(r,y)$ 
        \ENDIF
    \ENDFOR
    \RETURN $\bot$
    \end{algorithmic}
\end{algorithm}


In conclusion, an ordering $x_1,x_2,x_3,x_4$ of the $x$-coordinates of $E[3]$ and two subgroups of order $16$ with trivial intersection in $E[16]$, we obtain two polynomials in $X^8$ and $X^3$. It then suffices to take the $\gcd$ to find the value of the Weber invariant associated to this level structure of order $48$ on $E$: if $t$ is the output of $\GetWeberCube{}(p, j, j_1, j_2)$, we compute
\[ \gcd(X^3 - t, X^8 - \gamma_2)\]
and obtain $\frakf$-invariant associated to $E$. 

To find the $\ell$-isogenous $\frakf$-invariants corresponding to some $\ell$-isogeny $\varphi$ it suffices to evaluate the level structure in $\varphi$ and repeat the procedure depicted above. Since our algorithms relies on computing alternative paths through the Deuring correspondence (in particular, we never explicitly compute any $\ell$-isogenies), obtaining this evaluation requires the use of techniques which at this point are standard in isogeny-based cryptography. We refer to Section \ref{sec: implementation details} for details.

\section{Implementation results}
\label{sec: implementation}

We have implemented the algorithms \ModularEvaluationBigLevel{} and \ModularEvaluationBigChar{} in C++/NTL. 
The aim of this section is to give an overview of the code-base, before we present our implementation results.

\subsection{Implementation details}
\label{sec: implementation details}

We first describe the different parts of our library, which is available at:
\begin{center}
    {\smaller
    \url{https://github.com/tonioecto/modular-polynomial-computation-and-evaluation-from-supersingular-curves}
    }
\end{center}

\vspace{1ex}

\subsubsection*{NTL functionalities.} Our library uses NTL to perform basic arithmetic over finite fields $\FF_{p^{k}}$ (for $k \geq 1$), and can be compiled both using classes {\tt NTL::zz\_p} and {\tt NTL::zz\_pE} or {\tt NTL::ZZ\_p} and {\tt NTL::ZZ\_pE} if needed, integer arithmetic using class {\tt NTL::ZZ}, and polynomial arithmetic using classes {\tt NTL::zz\_pX}, {\tt NTL::zz\_pEX} or {\tt NTL::ZZ\_pX}, {\tt NTL::ZZ\_pEX} if needed. We furthermore use NTL vectors and matrices constructed from these classes.

\vspace{1ex}

\subsubsection*{Finding CRT primes.}
We precomputed a list of all primes up to $2^{25}$,
and use a segmented version of the sieve of Eratosthenes \cite{bays1977segmented}
to find larger primes if required (with an upper bound of $2^{50}$).
Based on the implementation of \cite{CMN21}, it takes an interval $[L,R]$
and finds all contained primes.
Sieving different intervals is independent, so the segmented sieve
parallelizes perfectly.

Instead of storing all primes and testing supersingularity of the
corresponding elliptic curves later, 
we run supersingularity tests on the fly after identifying all
primes in a sieving interval $[L,R]$.
Since the overwhelming majority of curves we test is ordinary,
our implementation benefits most from a test that discards
those curves quickly.
Following \cite{banegas2022efficient}, the best choice in our context is
Sutherland's supersingularity test based on determining whether
the 2-isogeny graph has a volcano structure (ordinary case)
or not (supersingular case) \cite{sutherland2012identifying}.
We implement Sutherland's test including the optimisations
proposed in \cite{banegas2022efficient}.

\vspace{1ex}

\subsubsection*{Field arithmetic.} We require arithmetic in $\FF_{p^{2k}}$ for small $k$. For all basic operations, we rely on NTL. Further, we precompute a matrix corresponding to the Frobenius action. Additionally, due to how the isogeny computations work, we also need a fixed, effective embedding $\iota_k : \FF_{p^2} \hookrightarrow \FF_{p^{2k}}$ for all $k$ used. This embedding is given by a fixed image $\iota_k(\omega)$, where $\omega$ is a fixed generator of $\FF_{p^2}$ of trace $0$.

For an element $a + b\omega \in \FF_{p^2}$ (represented in NTL as the tuple $(a,b)$), we can easily lift this to $\FF_{p^{2k}}$ by computing $a + \iota_k(\omega)b$ in $\FF_{p^{2k}}$. To coerce elements into $\FF_{p^2}$, we proceed as follows. Given an element $\alpha \in \FF_{p^{2k}}$ (represented in NTL as the tuple $(a_1, \dots, a_{2k})$), which we know is in the image of $\iota_k$, i.e., $\alpha = \iota_k(a + \omega b)$: 
\begin{enumerate}
    \item Recover $a$ from the trace as $a = \frac{\tr{\alpha}}{2k}$.
    \item Recover $b$ as $a_j/\iota_k(\omega)_j$, where $j > 1$ is any non-zero index of the representation $(\iota_k(\omega)_1, \dots, \iota_k(\omega)_{2k})$ of the element $\iota_k(\omega)$.
\end{enumerate}
We extended these lifting and coersion operations to the polynomial rings $\FF_{p^{2k}}[X]$, by lifting and coercing each coefficient, respectively.

\vspace{1ex}

\subsubsection*{Elliptic curves and isogenies.} We implement standard elliptic curve operations using Weierstraß curves. One of the main computational tasks required by the elliptic curve part of our library is the computation of smooth degree isogenies given by kernel generators of smooth order defined over small extensions $\FF_{p^{2k}}$ of~$\FF_{p^{2}}$. This setting is similar to that in \cite{eriksen2023deuring}, and thus the strategy employed is the same. Namely, we use supersingular curves $E$ with $\pi_E^2 = [-p]$, where $\pi_E$ denotes the $p$\nobreakdash-Frobenius endomorphism on $E$. As a result, all the isogenies will be defined over $\FF_{p^{2}}$, even if the kernel generators are not. Due to this, we can compute our isogenies by recovering the kernel polynomial using \cite[Algorithm 4]{eriksen2023deuring}, before employing Kohel's formula to recover the isogeny \cite[Chapter 2.4]{K96}.

Computing the endomorphism ring of a given supersingular elliptic curve~$E$ defined over~$\FF_p$
is done by brute-force searching an isogeny to a curve with known endomorphism ring.
Concretely,
we start by determining the $\FF_p$\nobreakdash-endomorphism
ring~$R$ of~$E$; the two cases $\ZZ[\pi]\cong\ZZ[\sqrt{-p}]$
and $\ZZ[\frac{1+\pi}2]\cong\ZZ[\frac{1+\sqrt{-p}}2]$
are readily distinguished by testing whether the rank
of $E(\FF_p)[2]$ equals $1$ or $3$ respectively.
This boils down to checking whether the cubic polynomial~$f(x)$
in the short Weierstraß equation $y^2=f(x)$ defining~$E$ 
splits over~$\FF_p$ or not.
Then, using a suitable starting curve~$E_0$
with the same $\FF_p$\nobreakdash-endomorphism ring
and known $\FF_{p^2}$-endomorphism ring,
there must exist an ideal class $[\fa]\in\Cl(R)$
such that $[\fa]\ast E_0=E$,
which we recover using a straightforward meet-in-the-middle
brute-force search using a (heuristic) generating set of small
generators for $\Cl(R)$. In practice, this means we do not set explicitly the constant $C$ from \cref{alg: endo ring}, we just increase the exponents until a solution is found.  

\vspace{1ex}

\subsubsection*{Lattices in quaternion algebras.} 
Our library identifies the quaternion algebra $B=H(-q,-p)$
simply by the pair $p,q\in\ZZ$, and the format
of its elements consists of a reference to $B$
together with a $5$\nobreakdash-tuple
of integers $t,x,y,z,d\in\ZZ$
representing the quaternion $(t+xi+yj+zk)/d$.
Arithmetic in this representation is straightforward.
We further represent a quaternion lattice~$I$
by a basis matrix in $\QQ^{4\times 4}$,
which is in turn represented as a matrix in $\ZZ^{4\times 4}$
together with a denominator in~$\ZZ$.
All basic functionality for quaternion lattices,
such as computing the intersection, product, or sum of two lattices,
computing left or right orders, or finding canonical or reduced lattice bases,
all essentially rely on the linear-algebra routines provided by NTL.

Our library also includes an implementation of a variant of the KLPT algorithm \cite{KLPT14}, which works for all primes where there exists an embedding $K \hookrightarrow B_{p, \infty}$, where $K$ is an imaginary quadratic field of class number $1$. Note that this is a very mild assumption on the prime $p$ (a quick heuristic estimate is that only $1$ in every $2^9$ primes fail this requirement). Our implementation uses the standard improvement due to Petit and Smith \cite{petit2018improvement} to reduce the output size. However, our implementation also differs in another way.
The output size (with the improvement from \cite{petit2018improvement}) of KLPT is typically stated as being in $\tilde{O}(p^3)$. However, this is based on the heuristic that the shortest prime-normed equivalent ideal is of norm $O(p^{1/2}\log(p))$. While this heuristic almost never fails in large characteristic, we are working in small enough characteristic to regularly encounter ideals where this fails. Hence, we change the size of the solution we search for on the fly to be in $\tilde{O}(p^2N_I^2)$, where $N_I$ denotes the norm of the shortest prime-normed equivalent ideal.\footnote{Note that we do have the guarantee that $N_I \in \tilde{O}(p)$, hence the size of the output size is still upper bounded by $\tilde{O}(p^4)$.}

\vspace{1ex}

\subsubsection*{Orders to j-invariants.}
In \OrdersToJSmall, we repeatedly apply the same procedure as in \cite{eriksen2023deuring} to translate orders to $j$-invariants. The only difference is that our implementation has precomputed elliptic curves $E/\QQ$ with CM by an order of class number $1$, together with the isogeny $\iota$ defining the complex multiplication. Thus, for the first step described in \cite{eriksen2023deuring}, we can take the starting curve $E_0/\FF_{p}$, to be a supersingular reduction of one of our precomputed curves. Note that the requirement on $p$ that one of the precomputed curves has supersingular reduction coincides with the requirement we described for KLPT (except for a constant number of primes $p$, where the precomputed curves might have bad reduction). Further, note that for $p \equiv 1 \pmod{12}$, the isomorphism $\O_0 \cong \End(E_0)$ cannot be determined solely by the (reduction of the) endomorphism $\iota$; see \cite[Section~3.1]{eriksen2023deuring} for details, and how to resolve this ambiguity.

\ModularEvaluationBigChar{} also requires the \OrdersToJBig{} algorithm introduced in \cite{leroux2023computation} to compute efficiently all the supersingular $j$-invariants and the corresponding maximal order. Our implementation follows quite closely the algorithm detailed by Leroux (see \cite[Algorithm 1]{leroux2023computation}). Maximal orders are represented by the 3 successive minima (that can be computed with LLL) of the trace\nobreakdash-0 sublattice.  

\vspace{1ex}

\subsubsection*{Computing Weber invariants.} To compute the Weber invariant of a curve with a given level structure, we implement the ideas discussed in \cref{sec: other modular}. In particular, we implement subroutines to evaluate 
multivariate polynomials of degree $2$ and $3$ (NTL only supports univariate polynomials), and 
we use built-in NTL functions for $\gcd$ computations and root finding. 
Furthermore, the algorithm requires computing the $\gcd$ of two precomputed resultants, which are stored in {\tt getresultants.cpp}. 
The memory cost of the hardcoded resultants is negligible in the context of the full algorithm, and the efficiency gain is significant.

To obtain the $\ell$-isogenous Weber invariants in \ModularEvaluationBigLevel{}, we also need to evaluate the level structure at the $\ell$-isogenies. Since no $\ell$-isogenies are ever explicitly computed, this is done by evaluating the level structure on a well-chosen endomorphism, before pushing it through the alternative isogeny path we have computed. Explicitly, this endomorphism equals the composition of the $\ell$\nobreakdash-isogeny with the dual of the alternative path, and can be found by multiplying the corresponding ideals.

Furthermore, \OrdersToJBig{} needs to be adapted to work with Weber invariants instead of $j$-invariants. For that, we represent Weber invariants as a $j$\nobreakdash-invariant together with some basis of the $48$-torsion, and with some coefficients to represent the level structure in the basis. By exploiting the symmetries coming from the Weber function, we obtain an efficient algorithm to recover the full list of Weber invariants.

\OrdersToJBig{} also needs to be adapted to work with Weber invariants instead of $j$-invariants, and so we need to modify \OrdersToJBig{} to compute all supersingular $\frakf$-invariants as well. For that, we will use our level structure parametrization of $\frakf$-invariants. That way, by using a basis of the 48-torsion, each Weber invariant can be fully derived from a set of small integers (giving the coefficients of the points of the level structure in the basis). Then, it suffices to propagate the 48-torsion basis throughout the isogenies involved in \OrdersToJBig{}.
For each supersingular $j$-invariant, we will then need to compute all $\frakf$-invariants above it from the basis.  
The symmetries coming from the Weber function allow us to be much more efficient to recover the full list of Weber invariants than one application of the algorithm we described in \cref{sec: other modular} for each of the 72 Weber invariants. 

\vspace{1ex}

\subsubsection*{The CRT method.} For the implementation of the CRT method in our library, we follow Sutherland's software {\tt classpoly} \cite{classpoly} based on \cite{enge2010class,sutherland2011computing}. In particular, our CRT functions in {\tt crt.cpp} take heavy inspiration from {\tt crt.c} in Sutherland's software. The main difference is that our algorithms use C++/NTL rather than C/GMP. We therefore defer a detailed explanation on the implementation of the CRT method to these previous works.

\vspace{1ex}

\subsubsection*{Polynomial interpolation.} In our library, we implement univariate polynomial interpolation following Algorithm 10.11 in \cite[\S 10]{moderncomputeralg}. In particular, we implement two main functions:
\begin{itemize}
    \item An algorithm to interpolate a univariate polynomial $f(X) \in \FF_{p^{2k}}[X]$ of degree $d$ at $d+1$ points $(x_0, f(x_0)), \dots, (x_d, f(x_d))$.
    \item An algorithm to interpolate a univariate polynomial $f(X) \in \FF_{p^{2k}}[X]$ of degree $d$ from its roots $(x_1,\dots, x_d)$. This is a special case of the first algorithm where $f(x_i) = 0$, but here we optimise for this case.
\end{itemize}
We remark that these algorithms outperform the in-built NTL interpolation function {\tt interpolate} and the NTL function {\tt BuildFromRoots}, respectively, for polynomials of relatively large degrees (which is our use case in this paper). 

\subsection{Results}
\label{sec: implementation results}

In this section, we present the performance of our implementation of the algorithms 
\ModularEvaluationBigLevel{} and \ModularEvaluationBigChar{}, and compare it to the 
state-of-the-art \cite{sutherland2013evaluation}. 

\vspace{1ex}

\subsubsection*{Performance analysis.}

All code was compiled and
run on a
 Linux server with
 two AMD~EPYC~9754 CPUs
 running at 2.25\,GHz
 (with dynamic frequency scaling and simultaneous multithreading disabled),
 1\,TB RAM,
 and
 using \texttt{g++} version~15.1.1
with flags \texttt{-std=c++17 -O3 -march=native -DNDEBUG}.

We ran experiments for the Weber variant of both algorithms \ModularEvaluationBigLevel{} and \ModularEvaluationBigChar{}
(i.e., we compute the Weber modular polynomial evaluated at a Weber invariant).
To understand the performance of these algorithms for 
varying level $\ell$, we fix the characteristic $p=2^{31}-1$, the Weber invariant $w=2$, 
and run both algorithms for the same set of growing levels~$\ell$. 
Since our implementation is heavily multithreaded,
to enable a meaningful comparison with data collected on other hardware,
we report the CPU \emph{core} time (i.e., in single-core equivalents)
rather than the wall-clock time consumed by each run of each algorithm.
Our results are summarized in \cref{tab:results}.

We see that despite a better asympotic complexity in $\ell$, we were not able to reach a point where \ModularEvaluationBigLevel{} is faster than \ModularEvaluationBigChar{}. This is the due to the fact that the hidden constants in the asymptotic complexity of the \OrderToJName{} algorithm are very large (due to the computation of a smooth isogeny of degree $\approx p^3$). 
Nonetheless, the current data is enough to witness that the scaling of \ModularEvaluationBigLevel{} and \ModularEvaluationBigChar{} are indeed quadratic and cubic in $\ell$ respectively.  


\begin{table}[ht!]
    \caption{Experimental data:\,
    CPU core time consumed by runs of \ModularEvaluationBigLevel{}
    and \ModularEvaluationBigChar{} for various levels $\ell$, 
    fixing the characteristic $p = 2^{31}-1$ and the input Weber invariant $w = 2$.}\label{tab:results}
    \centering 
    \renewcommand{\tabcolsep}{0.15cm}
    \renewcommand{\arraystretch}{1.2}
    \vspace{-1ex}
    \begin{tabular}{@{\;}r@{\quad}c@{\quad}c@{\quad}c@{\;}}
        \toprule
        Level $\ell$ & \ModularEvaluationBigLevel{} & \ModularEvaluationBigChar{} & Ratio \\
        \midrule
      211 &    7.68\,min    &  18.191\,s      &  25.3 \\
      419 &   24.01\,min    &  40.702\,s      &  35.4 \\
      607 &   51.34\,min    &    1.24\,min    &  41.3 \\
      811 &    1.46\,h      &    2.00\,min    &  43.7 \\
     1019 &    2.28\,h      &    3.07\,min    &  44.5 \\
     2003 &    9.12\,h      &   15.13\,min    &  36.2 \\
     3011 &   21.54\,h      &   46.53\,min    &  27.8 \\
     4003 &    1.63\,d      &    1.65\,h      &  23.7 \\
     5003 &    2.66\,d      &    3.31\,h      &  19.3 \\
     6007 &    3.89\,d      &    5.49\,h      &  17.0 \\
     7019 &    5.32\,d      &    8.42\,h      &  15.2 \\
     8011 &    7.00\,d      &   12.13\,h      &  13.8 \\
     9007 &    9.13\,d      &   17.93\,h      &  12.2 \\
    10007 &   11.39\,d      &    1.00\,d      &  11.3 \\
    11681 &   15.78\,d      &    1.56\,d      &  10.1 \\
        \bottomrule
    \end{tabular}
\end{table}

\begin{figure}
    \includegraphics[width=.99\textwidth]{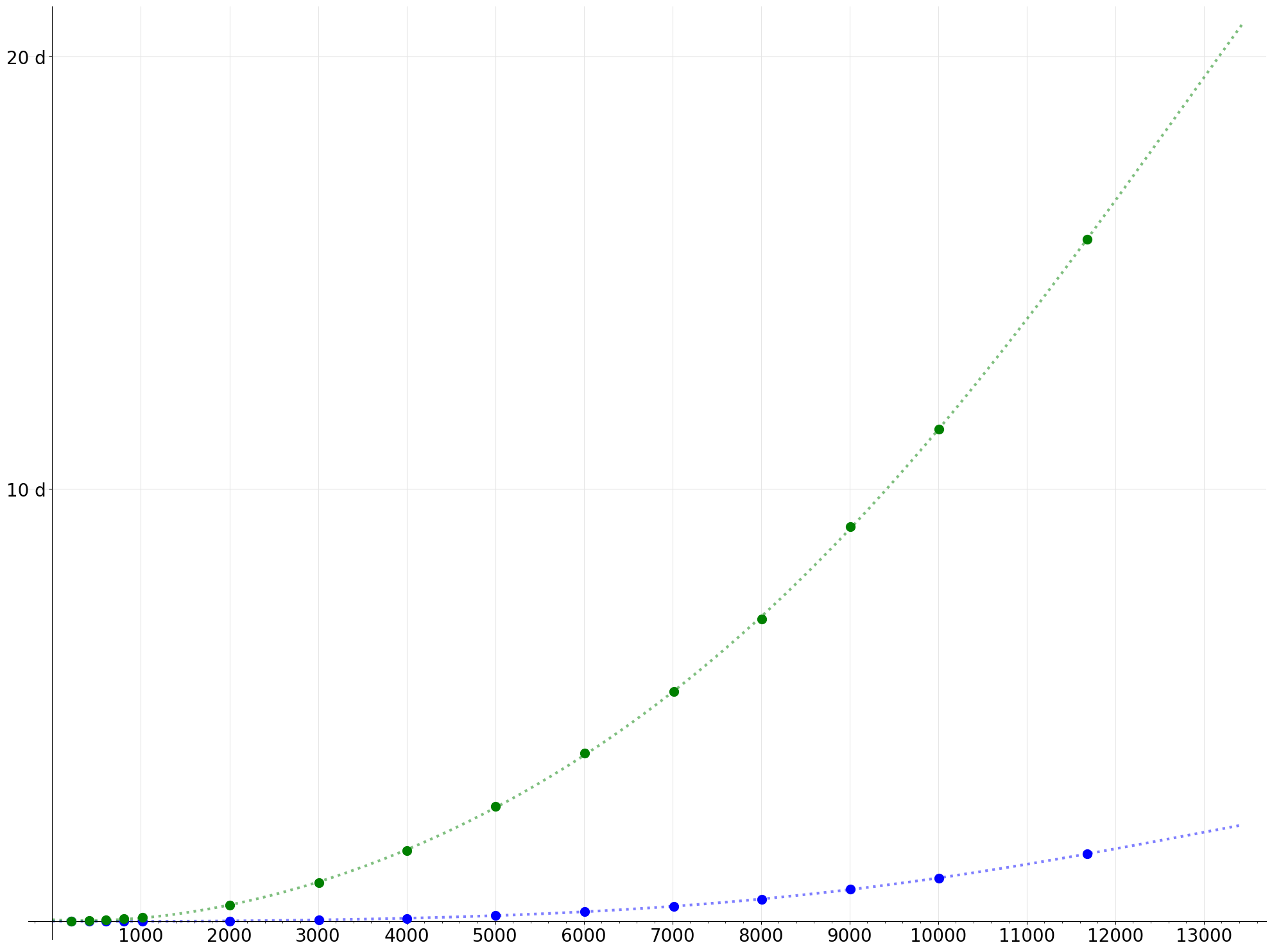}
    \caption{%
        \smaller
        Visualization of the data from \cref{tab:results}:
        The horizontal axis ranges through choices of~$\ell$,
        the vertical axis gives the corresponding runtime.
        Data points depicted in green are runs of \ModularEvaluationBigLevel{},
        whereas blue refers to \ModularEvaluationBigChar{}.
        \\
        The thin dotted lines are the result of fitting the given data points
        using SageMath's \texttt{find\_fit()} function
        with a polynomial of degree~$\leq 5$:
        Indeed, the resulting interpolation curves for
        \ModularEvaluationBigLevel{}
        and
        \ModularEvaluationBigChar{}
        are approximately quadratic and approximately cubic
        in~$\ell$,
        respectively,
        matching precisely their asymptotic runtime bounds.
        %
    }
    \label{fig:plot}
\end{figure}

 Extrapolating the results presented in \cref{tab:results},
we observe that our current implementation falls short of
a record computation (using a similar computational power as in \cite{sutherland2013evaluation}). Take, for instance, the large characteristic setting. In the SEA point-counting record described in \cite[Section 5]{sutherland2013evaluation} with $\log(p)\approx 16646$, the total time to compute evaluated modular polynomials using Weber invariants for level $\ell = 11681$ is reported to be around 2 CPU hours on a single thread. 

For the same level $\ell$, but a much smaller characteristic $p$, our implementation of \ModularEvaluationBigChar{} took around 40 CPU hours. The dependency on $\log p$ in the asymptotic complexity comes from the $\log p$ term in the height bound. When taking $\log(p)\approx 16646$ for $\ell = 11681$, one sees that the height bound is essentially doubled (compared to $\log p \approx 30$).
Thus, we can estimate that the same computation as Sutherland would have taken around $80$ CPU hours with our implementation, which is 1 or 2 orders of magnitude slower. Since the computations were run on different hardware, it is difficult to be more precise than that.  

\vspace{1ex}

\subsubsection*{Possible improvements.} There are still some places where significant improvements could be made. We identify the following possible improvements to our software (in order of expected impact):
\begin{enumerate}
    \item Currently, the quaternion operations are the bottleneck in \ModularEvaluationBigChar{} (taking up to 80\% percent of the global run time) despite being asymptotically negligible compared to interpolation. One explanation is that the hidden constants in these operations can be quite big (one quaternion multiplication is already 17 integer multiplications), and most operations on orders and ideals require to apply operations on bases of four quaternion elements. Another reason is that our implementation uses {\tt NTL::ZZ} which are integers of arbitrary size when the integers involved in all our quaternion computation can in fact be bounded. Thus, switching to fixed-sized integers should bring a noticeable improvement. Moreover, some of the operations on ideals use generic algorithms that could be improved with a tailored implementation. 
    \item Using $x$-only arithmetic for elliptic curve operations, and in particular in isogeny computations which are by far the bottleneck in \ModularEvaluationBigLevel{} due to the high constants involved in the complexity of computing the isogenies in \OrdersToJSmall.  
    \item Precomputing the action of each basis element of the endomorphism ring on the relevant $\ell^e$-torsion bases of the starting curve would also greatly improve the cost of computing the isogeny kernels in \OrdersToJSmall{} and thus positively impact the performances of \ModularEvaluationBigLevel{}.
\end{enumerate}

We are hopeful that the concrete
efficiency gains from these improvements could allow to significantly improve our implementation to the point where 
\ModularEvaluationBigChar{} could match or even beat the performances of the software from \cite{sutherland2013evaluation}. 

Depsite the better asymptotic scaling, it seems unlikely that the improvements listed above would be enough for \ModularEvaluationBigLevel{} to beat the other two algorithms for a realistic value of $\ell$. To reach that point, one would likely need a deeper algorithmic improvement allowing us to decrease the degree of the isogenies used in \OrderToJName.  

We remark, however, that all the changes listed above require a significant implementation effort, and 
are thus out of scope for this paper.

\section{Conclusion}

We have introduced several new algorithms to evaluate modular polynomials using supersingular curves and the Deuring correspondence. These algorithms all improve the asymptotic time or space complexity of previously known algorithms for some range of inputs. 

We implemented our algorithms in C++, exhibiting the practicality of the algorithms presented. We further suggest 
potential improvements that could lead to important practical speed-ups of our implementation, but they require an extensive optimization effort and are thus left for future work.


\bibliographystyle{amsplain}
\bibliography{biblio}

\providecommand{\bysame}{\leavevmode\hbox to3em{\hrulefill}\thinspace}
\providecommand{\MR}{\relax\ifhmode\unskip\space\fi MR }
\providecommand{\MRhref}[2]{%
  \href{http://www.ams.org/mathscinet-getitem?mr=#1}{#2}
}
\providecommand{\href}[2]{#2}
\begin{thebibliography}{10}

\bibitem{ankeny1952least}
Nesmith~Cornett Ankeny, \emph{The least quadratic non residue}, Annals of Mathematics (1952), 65--72.

\bibitem{arpin2023finding}
Sarah Arpin, James Clements, Pierrick Dartois, Jonathan~Komada Eriksen, P{\'{e}}ter Kutas, and Benjamin Wesolowski, \emph{Finding orientations of supersingular elliptic curves and quaternion orders}, Des. Codes Cryptogr. \textbf{92} (2024), no.~11, 3447--3493.

\bibitem{banegas2022efficient}
Gustavo Banegas, Valerie Gilchrist, and Benjamin Smith, \emph{Efficient supersingularity testing over $\mathbb{F}_p$ and {CSIDH} key validation}, Mathematical Cryptology \textbf{2} (2022), no.~1, 21--35.

\bibitem{bays1977segmented}
Carter Bays and Richard~H. Hudson, \emph{The segmented sieve of {Eratosthenes} and primes in arithmetic progressions to 1012}, BIT Numerical Mathematics \textbf{17} (1977), no.~2, 121--127.

\bibitem{BFLS20}
Daniel~J. Bernstein, Luca {De Feo}, Antonin Leroux, and Benjamin Smith, \emph{Faster computation of isogenies of large prime degree}, ANTS (2020).

\bibitem{blake1999computation}
Ian~F. Blake, J{\'a}nos~A Csirik, Michael Rubinstein, and Gadiel Seroussi, \emph{On the computation of modular polynomials for elliptic curves}, HP Laboratories Technical Report (1999).

\bibitem{magma}
Wieb Bosma, John Cannon, and Catherine Playoust, \emph{The {M}agma algebra system. {I}. {T}he user language}, J. Symbolic Comput. \textbf{24} (1997), no.~3-4, 235--265, Computational algebra and number theory (London, 1993). \url{https://www.math.ru.nl/~bosma/pubs/JSC1997Magma.pdf}. \MR{MR1484478}

\bibitem{broker2012modular}
Reinier Br{\"o}ker, Kristin Lauter, and Andrew Sutherland, \emph{Modular polynomials via isogeny volcanoes}, Mathematics of Computation \textbf{81} (2012), no.~278, 1201--1231.

\bibitem{brokerthesis}
Reinier~Martijn Br{\"o}ker, \emph{Constructing elliptic curves of prescribed order}, Ph.D. thesis, Leiden University, 2006.

\bibitem{CLMPR18}
Wouter Castryck, Tanja Lange, Chloe Martindale, Lorenz Panny, and Joost Renes, \emph{{CSIDH}: an efficient post-quantum commutative group action}, ASIACRYPT~2018, Springer, 2018, pp.~395--427.

\bibitem{charles2005computing}
Denis Charles and Kristin Lauter, \emph{Computing modular polynomials}, LMS Journal of Computation and Mathematics \textbf{8} (2005), 195--204.

\bibitem{CK19}
Leonardo Col{\`o} and David Kohel, \emph{Orienting supersingular isogeny graphs}, Number-Theoretic Methods in Cryptology 2019 (2019).

\bibitem{splitsearcher}
Maria {Corte-Real Santos}, Craig Costello, and Sam Frengley, \emph{An algorithm for efficient detection of {$(N,N)$}-splittings and its application to the isogeny problem in dimension 2}, {PKC}~2024. {P}art~{III}, Lecture Notes in Computer Science, vol. 14603, Springer, 2024, pp.~157--189. \MR{4763489}

\bibitem{CMN21}
Craig Costello, Michael Meyer, and Michael Naehrig, \emph{Sieving for twin smooth integers with solutions to the {Prouhet-Tarry-Escott} problem}, {EUROCRYPT}~2021, Lecture Notes in Computer Science, vol. 12696, Springer, 2021, pp.~272--301.

\bibitem{couveignes}
Jean~Marc Couveignes, \emph{Hard homogeneous spaces}, IACR Cryptology ePrint Archive, 2006.

\bibitem{MR2290010}
David Cox, John Little, and Donal O'Shea, \emph{Ideals, varieties, and algorithms}, 3 ed., Undergraduate Texts in Mathematics, Springer, 2007. \MR{2290010}

\bibitem{delfs2016computing}
Christina Delfs and Steven~D. Galbraith, \emph{Computing isogenies between supersingular elliptic curves over {$\mathbb F_p$}}, Designs, Codes and Cryptography \textbf{78} (2016), 425--440.

\bibitem{EHLMP18}
Kirsten Eisentr{\"a}ger, Sean Hallgren, Kristin Lauter, Travis Morrison, and Christophe Petit, \emph{Supersingular isogeny graphs and endomorphism rings: Reductions and solutions}, EUROCRYPT 2018, Springer, 2018, pp.~329--368.

\bibitem{elkies1998elliptic}
Noam~D. Elkies et~al., \emph{Elliptic and modular curves over finite fields and related computational issues}, AMS/IP Studies in Advanced Mathematics \textbf{7} (1998), 21--76.

\bibitem{enge2009complexity}
Andreas Enge, \emph{The complexity of class polynomial computation via floating point approximations}, Mathematics of Computation \textbf{78} (2009), no.~266, 1089--1107.

\bibitem{enge2010class}
Andreas Enge and Andrew~V. Sutherland, \emph{Class invariants by the {CRT} method}, International Algorithmic Number Theory Symposium, Springer, 2010, pp.~142--156.

\bibitem{eriksen2023deuring}
Jonathan~Komada Eriksen, Lorenz Panny, Jana Sot{\'{a}}kov{\'{a}}, and Mattia Veroni, \emph{{Deuring} for the people: Supersingular elliptic curves with prescribed endomorphism ring in general characteristic}, LuCaNT 2023, 2023.

\bibitem{GPS17}
Steven~D. Galbraith, Christophe Petit, and Javier Silva, \emph{Identification protocols and signature schemes based on supersingular isogeny problems}, ASIACRYPT~2017, 2017.

\bibitem{S09}
Joseph H.~Silverman, \emph{The arithmetic of elliptic curves}, 2 ed., Graduate Texts in Mathematics, vol. 106, Springer, 2009.

\bibitem{KLPT14}
David Kohel, Kristin~E. Lauter, Christophe Petit, and Jean-Pierre Tignol, \emph{On the quaternion $\ell$-isogeny path problem}, 2014.

\bibitem{K96}
David~R. Kohel, \emph{Endomorphism rings of elliptic curves over finite fields}, Ph.D. thesis, University of California at Berkeley, 1996.

\bibitem{kunzweiler2024computing}
Sabrina Kunzweiler and Damien Robert, \emph{Computing modular polynomials by deformation}, 2024.

\bibitem{lang2006frobenius}
Serge Lang and Hale Trotter, \emph{Frobenius distributions in $\mathrm{GL}_2$\nobreakdash-extensions: Distribution of {Frobenius} automorphisms in $\mathrm{GL}_2$\nobreakdash-extensions of the rational numbers}, Lecture Notes in Mathematics, vol. 504, Springer, 2006.

\bibitem{lehmann1994counting}
Frank Lehmann, Markus Maurer, Volker M{\"u}ller, and Victor Shoup, \emph{Counting the number of points on elliptic curves over finite fields of characteristic greater than three}, International Algorithmic Number Theory Symposium, Springer, 1994, pp.~60--70.

\bibitem{leroux2022quaternion}
Antonin Leroux, \emph{Quaternion algebra and isogeny-based cryptography}, Ph.D. thesis, Ecole doctorale de l’Institut Polytechnique de Paris, 2022.

\bibitem{leroux2023computation}
Antonin Leroux, \emph{Computation of {Hilbert} class polynomials and modular polynomials from supersingular elliptic curves}, 2023.

\bibitem{lmfdb}
The {LMFDB Collaboration}, \emph{The {L}-functions and modular forms database}, \url{https://www.lmfdb.org}, 2024, [Online; accessed 7 June 2024].

\bibitem{morain1995calcul}
Fran{\c{c}}ois Morain, \emph{Calcul du nombre de points sur une courbe elliptique dans un corps fini: aspects algorithmiques}, Journal de th{\'e}orie des nombres de Bordeaux \textbf{7} (1995), no.~1, 255--282.

\bibitem{petit2018improvement}
Christophe Petit and Spike Smith, \emph{An improvement to the quaternion analogue of the $\ell$\nobreakdash-isogeny problem}, Presentation at MathCrypt (2018).

\bibitem{robert2022some}
Damien Robert, \emph{Some applications of higher dimensional isogenies to elliptic curves (overview of results)}, IACR Cryptology ePrint Archive, 2022.

\bibitem{RS06}
Alexander Rostovtsev and Anton Stolbunov, \emph{Public-key cryptosystem based on isogenies}, IACR Cryptology ePrint Archive, 2006.

\bibitem{corte2022accelerating}
Maria~Corte{-}Real Santos, Craig Costello, and Jia Shi, \emph{Accelerating the {Delfs-Galbraith} algorithm with fast subfield root detection}, {CRYPTO}~{(3)}, Lecture Notes in Computer Science, vol. 13509, Springer, 2022, pp.~285--314.

\bibitem{ntl}
Victor Shoup et~al., \emph{{NTL}: A library for doing number theory}, 2001.

\bibitem{classpoly}
A.~V. Sutherland, \emph{{\tt classpoly}}, Software package, version 1.0.3, \url{https://math.mit.edu/~drew/classpoly.html}, accessed 10 January 2025.

\bibitem{sutherland2011computing}
Andrew Sutherland, \emph{Computing {Hilbert} class polynomials with the {Chinese} remainder theorem}, Mathematics of Computation \textbf{80} (2011), no.~273, 501--538.

\bibitem{sutherland2013evaluation}
\bysame, \emph{On the evaluation of modular polynomials}, ANTS X, vol.~1, The Open Book Series, no.~1, Mathematical Sciences Publishers, 2013, pp.~531--555.

\bibitem{sutherland2012identifying}
Andrew~V. Sutherland, \emph{Identifying supersingular elliptic curves}, LMS Journal of Computation and Mathematics \textbf{15} (2012), 317--325.

\bibitem{sage}
{The Sage Developers}, \emph{{SageMath}, the {S}age {M}athematics {S}oftware {S}ystem ({V}ersion 10.0)}, 2024, {\tt https://www.sagemath.org}.

\bibitem{V71}
Jacques V{\'e}lu, \emph{Isog{\'e}nies entre courbes elliptiques}, Comptes-Rendus de l'Acad{\'e}mie des Sciences, S{\'e}rie I \textbf{273} (1971), 238--241.

\bibitem{voight}
John Voight, \emph{Quaternion algebras}, Graduate Texts in Mathematics, Springer, 2018.

\bibitem{moderncomputeralg}
Joachim von~zur Gathen and J\"urgen Gerhard, \emph{Modern computer algebra}, Cambridge University Press, New York, 1999.

\end{thebibliography}

\appendix
\section{Maps used for Weber invariant computation}\label{appendix:mod-curve-formulae}

We recall the formulas for all the maps involved in the computation described in \cref{sec: other modular}. 
Some of these maps can be found on the LMFDB, and the other maps were computed using MAGMA \cite{magma}. 

\bigskip

\noindent
\resizebox{.99\textwidth}{!}{%
\begin{minipage}{1.26\textwidth}%
\begin{equation*}
    \begin{aligned}
         G_0(16)(X) \colonequals \, & X^{24} + 768X^{23} + 213504X^{22} + 25587712X^{21} + 1285091328X^{20} + \\  
    & 37065719808X^{19} + 714316447744X^{18} + 9993958981632X^{17} + 
    106811460943872X^{16} + \\
    & 901713888280576X^{15} + 6152208865296384X^{14} + 
    34465746750799872X^{13} + \\ 
    & 160256972254347264X^{12} +  622699889175822336X^{11} + 2028993128165277696X^{10} + \\ 
    & 5546125766502121472X^9 + 12682896313210109952X^8 +
    24114242729778610176X^7 + \\
    & 37724965228622381056X^6 +  
    47787695646028333056X^5 + 47841738841556779008X^4 + \\ 
    &36461142583191535616X^3 + 19887895954468110336X^2 + 
    6917529027641081856X  + \\ 
    & 1152921504606846976 \\
    \allowdisplaybreaks
    H_0 (16)(X) \colonequals \, & X^{23} + 22X^{22} + 208X^{21} 
    + 1104X^{20} + 3584X^{19} + 7168X^{18} + 8192X^{17} + 4096X^{16} \\
    \allowdisplaybreaks
    I_0(16)(X) \colonequals \, & X^{24} + 48X^{23} + 1104X^{22} + 16192X^{21} + 169968X^{20} + 1358208X^{19} + \\
    & 8577664X^{18} + 43863552X^{17} + 184560432X^{16} + 645654016X^{15} + \\
    & 1889647104X^{14} + 4639475712X^{13} + 9552059904X^{12} + 16434548736X^{11} + \\
    & 23468814336X^{10} + 27518124032X^9 + 26077016832X^8 + 19519451136X^7 + \\
    & 11161194496X^6 + 4633214976X^5 + 1285091328X^4 + \\ 
    & 204701696X^3 + 13664256X^2 + 393216X + 4096  \\
    J_0(16)(X) \colonequals \, & X^8 + 16X^7 + 112X^6 + 448X^5 + 
    1104X^4 + 1664X^3 + 1408X^2 + 512X \\
    \allowdisplaybreaks
    G_s(16)(X,Y) \colonequals \, & -16X^7Y^{11} - X^7Y^3 + 16X^6Y^{10} + X^6Y^2 - 16X^5Y^9 
    - X^5Y + 16X^4Y^8 + X^4 \\ 
    H_s(16)(X,Y) \colonequals \, & 4Y^{12} \\
    \allowdisplaybreaks
    I_s(16)(X,Y) \colonequals \, & -65536X^7Y^{32} + 24576X^7Y^{28} - 12288X^7Y^{24} + 
    3200X^7Y^{20} - 736X^7Y^{16} + \\
    & 120X^7Y^{12} - 14X^7Y^8 + 
    X^7Y^4 + 65536X^6Y^{32} - 16384X^6Y^{28} + 9216X^6Y^{24} - \\ 
    & 1792X^6Y^{20} + 384X^6Y^{16} - 48X^6Y^{12} + 4X^6Y^8 - 
    65536X^5Y^{32} + 8192X^5Y^{28} - \\ 
    &7168X^5Y^{24} +  768X^5Y^{20} - 192X^5Y^{16} + 16X^5Y^{12} + 65536X^4Y^{32} + 6144X^4Y^{24} + \\
    & 128X^4Y^{16} - 65536X^3Y^{32} + 
    57344X^3Y^{28} - 30720X^3Y^{24} +  11648X^3Y^{20} - \\ 
    & 3328X^3Y^{16} + 728X^3Y^{12} - 120X^3Y^8 + 14X^3Y^4 - 
    X^3 + 65536X^2Y^{32} - 49152X^2Y^{28} \\
    & + 23552X^2Y^{24} - 7936X^2Y^{20} + 1984X^2Y^{16} - 368X^2Y^{12} + 48X^2Y^8 - \\
    & 4X^2Y^4 - 65536XY^{32} + 40960XY^{28} - 
     17408XY^{24} + 5120XY^{20} - 1088XY^{16} + \\
     & 160XY^{12} - 16XY^8 + 
    65536Y^{32} - 32768Y^{28} + 12288Y^{24} - 3072Y^{20} + 512Y^{16} - 
    64Y^{12} \\ 
    \allowdisplaybreaks
    J_s(16)(X,Y) \colonequals \, & 8Y^{16} + Y^8 \\
    \allowdisplaybreaks
    F_s(16)(X,Y) \colonequals \, & 6X^8Y^{12}+X^8Y^4-16X^4Y^8-16Y^{12}-X^4.
    \end{aligned}
\end{equation*}%
\end{minipage}%
}

\vfill\pagebreak    

\end{document}